\documentclass[11pt,a4paper,reqno]{article}

\usepackage[margin=2.9cm]{geometry}
\usepackage[table]{xcolor}
\usepackage{tikz}
\usepackage{color}
\usepackage{amssymb}
\usepackage{float}
\usepackage{enumitem}
\usepackage{array}
\usepackage{longtable}
\usepackage{ctable} 
\usepackage{multirow}
\usepackage{mathrsfs}
\usepackage{caption}
\usepackage{footmisc}
\usepackage{slashbox}
\usepackage{hyperref}
\usepackage{amsthm,amssymb,amsmath}
\usepackage{tabularx}
\usepackage{hhline}
\usepackage{multicol}
\usepackage{calligra}

\usepackage{hhline}

\usepackage{cite}

\newcolumntype{C}[1]{>{\centering\arraybackslash}m{#1}}
\newcolumntype{L}[1]{>{\arraybackslash}m{#1}}
\newcolumntype{D}[1]{>{\centering\arraybackslash}m{#1}}
\newcolumntype{E}[1]{>{\arraybackslash}m{#1}}

\theoremstyle{definition}

\newtheorem{theorem}{Theorem}[section]
\newtheorem{corollary}[theorem]{Corollary}
\newtheorem{lemma}[theorem]{Lemma}
\newtheorem{definition}[theorem]{Definition}
\newtheorem{proposition}[theorem]{Proposition}
\newtheorem{remark}[theorem]{Remark}
\newtheorem{notation}[theorem]{Notation}

\newtheorem{example}[theorem]{Example}

\def\F{\mathbb{F}}
\def\Fq{\F_q}
\def\Fqn{\Fq^n}
\def\Fqm{\F_{q^m}}
\def\Fqmn{\Fqm^n}
\def\rk{\textup{rk}}
\def\mP{\mathscr{P}}
\def\mL{\mathscr{L}}
\def\At{\textup{At}}

\def\Z{\mathbb{Z}}
\def\cov{\mathrel{<\kern-.6em\raise.015ex\hbox{$\cdot$}}}
\newcommand{\qbin}[2]{\begin{bmatrix}{#1}\\ {#2}\end{bmatrix}_q}
\newcommand{\qmbin}[2]{\begin{bmatrix}{#1}\\ {#2}\end{bmatrix}_{q^m}}
\newcommand{\Qbin}[2]{\begin{bmatrix}{#1}\\ {#2}\end{bmatrix}_Q}

\newcommand{\qbinomial}[3]{\begin{bmatrix}{#1}\\ {#2}\end{bmatrix}_{#3}}
\newcommand{\bbin}[3]{\begin{bmatrix}{#1}\\ {#2}\end{bmatrix}_{q^{#3}}}
\def\<{\left<}
\def\>{\right>}

\def\crit{\textup{crit}}
\def\H{\textup{H}}

\def\F{\mathbb{F}}

\def\mH{\mathscr{H}}
\def\wt{\textup{wt}}

\newcolumntype{C}[1]{>{\centering\arraybackslash}m{#1}}
\newcolumntype{L}[1]{>{\arraybackslash}m{#1}}

\definecolor{myblue}{HTML}{0ABAB5}
\def\cb{\cellcolor{myblue!15}}

\setitemize{itemsep=0.2em}
\setenumerate{itemsep=0.2em}

\title{\textbf{Rank-Metric Lattices}}

\usepackage{authblk}

\author[1]{Giuseppe Cotardo\thanks{G. C. is supported by the Irish Research Council, grant n. GOIPG/2018/2534.}}
\affil[1]{School of Mathematics and Statistics, University College Dublin, Ireland}

\author[2]{Alberto Ravagnani\thanks{A. R. is supported by the Dutch Research Council through grants VI.Vidi.203.045, OCENW.KLEIN.539, 
and by the Royal Academy of Arts and Sciences of the Netherlands.}}
\affil[2]{Department of Mathematics and Computer Science, Eindhoven University of Technology, the Netherlands}

\date{}

\usepackage{setspace}
\makeatletter
\newcommand{\subjclass}[2][1991]{%
  \let\@oldtitle\@title%
  \gdef\@title{\@oldtitle\footnotetext{#1 \emph{Mathematics subject classification.} #2}}%
}
\newcommand{\keywords}[1]{%
  \let\@@oldtitle\@title%
  \gdef\@title{\@@oldtitle\footnotetext{\emph{Key words and phrases.} #1.}}%
}
\makeatother

\begin{document}
\maketitle

\begin{abstract}
We introduce the class of rank-metric geometric lattices and initiate the study of their structural properties.
Rank-metric lattices  can be seen as the $q$-analogues of higher-weight Dowling lattices, defined by Dowling himself in 1971. We fully characterize the
supersolvable
rank-metric lattices and compute their characteristic polynomials.
We then concentrate on the smallest rank-metric lattice whose characteristic polynomial we cannot compute, and provide a formula for it under a polynomiality assumption on its Whitney numbers of the first kind.
The proof relies on computational results and on the theory of vector rank-metric codes, which we review in this paper
from the perspective of rank-metric lattices. More precisely, 
we introduce the notion of lattice-rank weights of a rank-metric code and investigate their properties as combinatorial invariants and as code distinguishers for inequivalent codes.
\end{abstract}

\bigskip

\section*{Introduction}

This paper initiates the study of \textit{rank-metric lattices} and of their connection with the invariant theory of rank-metric codes.
We define a rank-metric lattice as the combinatorial subgeometry of $\smash{\F_{q^m}^n}$ whose atoms are the projective points corresponding to vectors having small rank (defined later).

Rank-metric lattices can be seen as the $q$-analogues of a family of lattices that have been studied since the seventies, namely higher-weight Dowling lattices. These were introduced by Dowling himself~\cite{dowling1971codes,dowling1973q} in connection with the main problem of classical coding theory, i.e., computing the largest dimension of an error-correcting code in the Hamming metric with given length and correction capability. This question connects with the famous conjecture by Segre on the largest size of an arc in a projective space over a finite field~\cite{segre1955curve}.
Higher-weight Dowling lattices have been studied by various authors~\cite{dowling1971codes,dowling1973q,bonin1993automorphism,bonin1993modular,kung1996critical,brini1982some,games1983packing,zaslavsky1987mobius,ravagnani2019whitney,britz2005extensions}, yet the techniques for computing their characteristic polynomial are to date unknown.

Computing the characteristic polynomials of rank-metric lattices is equivalent to counting/estimating the number of rank-metric codes with given length and correction capability, a line of research that is becoming increasingly important in contemporary coding theory; see e.g.~\cite{antrobus2019maximal,byrne2020partition,gruica2020common,neri2018genericity}.

The first member of the family of rank-metric lattices is precisely the $q$-analogue of the first member of the family of higher-weight Dowling lattices. Indeed, the former is (isomorphic to) the lattice of subspaces of $\F_q^n$, and the latter is the lattice of subsets of $\{1,\ldots,n\}$. It is natural to ask to which extent this analogy holds.
In this paper, we show that divergences in the behavior of $q$-analogous lattices already appear when examining the second family members. For example, while the second 
higher-Weight Dowling lattices is supersolvable (it is known as the \textit{$q$-analogue of the partition lattice}),
the second rank-metric lattice is not.
In fact, in Section~\ref{sec:supersolvable} of this paper we provide a complete characterization of the modular elements of rank-metric lattices, and determine all the supersolvable ones. The proof techniques follow to some extent Bonin's approach~\cite{bonin1993modular}, but there are important technical differences in the arguments, which almost never translate from
higher-weight Dowling lattices to rank-metric lattices (we will explain why).

In Section~\ref{secPAP:5} we compute the characteristic polynomial of some rank-metric lattices, and establish general properties of their roots by combining the characterization of modular elements mentioned above with Stanley's modular factorization theorem~\cite{stanley1971modular}.
The smallest rank-metric lattices whose characteristic polynomial we cannot compute is the one generated by vectors in $\smash{\F_{q^4}^4}$ of rank at most $2$. Using a computational approach based on coding theory, in Section~\ref{sec:specific_chi} we provide a closed formula for its characteristic polynomial under the assumption that its second Whitney number is a polynomial in $q$. The formula we obtain is compatible with all the other properties of rank-metric lattices we establish throughout the paper.

In Section~\ref{sec:last} we consider applications of rank-metric lattices to the invariant theory of (vector) rank-metric codes. We show how rank-metric lattices can be used to define new distinguishers for this class of codes. We start from the observation that the elements of the  first rank-metric lattice are exactly the \textit{optimal rank-metric anticodes} defined in \cite{ravagnani2016generalized}
in connection with a class of invariants for (vector) rank-metric codes, namely the \textit{generalized vector rank weights}; see~\cite{kurihara2015relative}. In this paper, we extend the definition of
\cite{ravagnani2016generalized}
to ``higher weight'' rank-metric lattices.
 We then classify some families of rank-metric codes that show interesting extremality properties with respect to these notions
 and study the concepts 
 \textit{lattice weight distribution} and \textit{lattice binomial moments} in this context.

\section{Posets and Lattices}

In this section we recall some definitions and results on posets and lattices that we will need throughout the paper. A standard reference for this part is \cite[Chapter~3]{stanley2011enumerative}. The reader who is already familiar with lattice theory can skip this section.

In the sequel, $(\mP,\leq)$ denotes a finite poset. We say that $(\mP,\leq)$ is \textbf{trivial} if the cardinality of $\mP$ is $1$.
We sometimes abuse notation
and simply write $\mP$ for 
$(\mP,\leq)$. Moreover, for  $s,t\in\mP$ we write $s<t$ for $s\leq t$ and $s\neq t$.

\begin{definition}
	 Elements $s,t \in \mP$ 
	 are \textbf{comparable} if $s\leq t$ or $t\leq s$. We say that $t$ \textbf{covers} $s$ (or that $s$ \textbf{is covered} by $t$)  if $s<t$ and there is no element $u \in \mP$ such that $s<u<t$.
	 The notation is $s\cov t$.
\end{definition}

\begin{definition}
The \textbf{join} of $s,t \in \mP$, should it exist, is 
the element $u \in \mP$ that satisfies $s \le u$, $t \le u$,
and $u \le v$ for all $v \in \mP$ with $s \le v$ and $t \le v$.
Dually, the \textbf{meet} of $s,t \in \mP$, should it exist, is 
the element $u \in \mP$ that satisfies $u \le s$, $u \le t$,
and $v \le u$ for all $v \in \mP$ with $v \le s$ and $v \le t$. The poset $\mP$ has \textbf{minimum element} $0 \in \mP$
    if $0 \le s$ for all $s \in \mP$. Dually, it has \textbf{maximum element} $1 \in \mP$ if $s \le 1$ for all $s \in \mP$.
\end{definition}

\begin{notation}
    It is easy to check that the join (resp., the meet) of $s,t \in \mP$, should it exist,
    is unique and we denote it by $s \vee t$ (resp., by $s \wedge t$).
    The minimum (resp., the maximum) element of $\mP$, should it exist,
is unique and it is denoted by $0_{\mP}$ (resp., by $1_{\mP}$), or by $0$ (resp., $1$) if the poset is clear from context.
\end{notation}

\begin{definition}
\label{def:mobius}
	The \textbf{M\"obius function} of $\mP$ is defined recursively by 
	$\mu_\mP(s,t)=1$ if $s=t$,
	$\mu_\mP(s,t)=-\sum_{s \le u <t} \mu_\mP(s,u)$ if $s<t$, and
	$\mu_\mP(s,t)=0$ otherwise.
We sometimes  write $\mu_\mP(t)$ for $\mu_{\mP}(0,t)$, when $\mP$ has minimum element $0$.
\end{definition}

The following result is a cornerstone of combinatorics.

\begin{proposition}[M\"obius Inversion Formula] \label{prop:mobinv}
	Let $\mathbb{K}$ be a field and let $f:\mP\longrightarrow\mathbb{K}$ be a function. Define $g:\mP\longrightarrow\mathbb{K}$
	by $g(t)=\sum_{s\leq t}f(s)$ for all $t \in \mP$.
	Then
	\begin{equation*}
	f(t)=\sum_{s\leq t}\mu_\mP(s,t)\, g(s) \quad \mbox{for all $t \in \mP$}.
	\end{equation*}
\end{proposition}

\textit{Chains} play a crucial role in determining the structure of a poset. They are defined as follows.

\begin{definition}
A \textbf{chain} in $\mP$ is a non-empty subset $C \subseteq \mP$ such that every $s,t \in C$ are comparable. We say that $C$ is \textbf{maximal} if there is no chain in $\mP$ that strictly contains $C$.
We say that $\mP$ is \textbf{graded} if all maximal chains in $\mP$ have the same cardinality.
\end{definition}

We will also
need the following concepts.

\begin{definition} \label{def:modular}
Suppose that $\mP$ is graded and with minimum element $0$. The \textbf{rank function} $\rho_\mP: \mP \to \Z$ of $\mP$ is uniquely defined
by $\rho_\mP(0)=0$ and $\rho_\mP(t)=\rho_\mP(s)+1$ whenever $s,t \in \mP$ and $s \cov t$. In this case we call $\rho_\mP(s)$ the \textbf{rank} of $s$ and the elements $s \in \mP$ with $\rho_\mP(s)=1$ are called \textbf{atoms} of $\mP$. We denote the set of atoms of $\mP$ by $\At(\mP) \subseteq \mP$. We call
$\mP$ \textbf{semimodular} if it is graded and 
\begin{equation}
    \label{eq:semimodrel}
    \rho_\mP(s\wedge t)+\rho_\mP(s\vee t)\leq \rho_\mP(s)+\rho_\mP(t) \quad \mbox{for all $s,t \in \mP$.}
\end{equation}
An element $t \in \mP$ is called \textbf{modular} if equality holds in 
\eqref{eq:semimodrel} for all $s \in \mP$.
\end{definition}

In this paper, we focus on the following special family of posets, called \textit{lattices}. Standard examples of posets (such as the poset of subsets of a finite set and the poset of linear subspaces of a space over a finite field) are in fact lattices.

\begin{definition}
	A \textbf{finite lattice} $(\mL, \le)$ is
	a finite poset where every $s,t \in \mL$ have a join and a meet.
	In this case, join and meet can be seen as commutative and associative operations
	$\vee, \wedge : \mL \times \mL \to \mL$.
	In particular, the \textbf{join} (resp., the \textbf{meet}) of a non-empty subset $S \subseteq \mL$ is well-defined as the join (resp., the meet) of its elements and denoted by~$\vee S$ (resp., by~$\wedge S$). Furthermore, $\mL$ has a minimum and maximum element ($0_\mL$ and~$1_\mL$, resp.). Finally, if $\mL$ is graded then the \textbf{rank} of $\mL$ is $\rk(\mL)=\rho_\mL(1_\mL)$.
\end{definition}

One of the most studied combinatorial invariants of a lattice is its characteristic polynomial, defined as follows.

\begin{definition}
    Let $\mL$ be a finite graded lattice with rank function $\rho$ and M\"obius function~$\mu$. The \textbf{characteristic polynomial} of $\mL$ is the element of $\Z[\lambda]$ defined as
    \begin{equation*}
        \chi(\mL;\lambda):=\sum_{s\in\mL}\mu(s)\lambda^{\rk(\mL)-\rho(s)}=\sum_{i=0}^{\rk(\mL)}w_i(\mL)\lambda^{\rk(\mL)-i},
    \end{equation*}
    where
    \begin{equation*}
        w_i(\mL)=\sum_{\substack{s\in\mL\\\rho(s)=i}}\mu(s)
    \end{equation*}
    is the $i$-\textbf{th Whitney number of the first kind} of $\mL$. The  $i$-\textbf{th Whitney number of the second kind} of $\mL$, denoted by $W_i(\mL)$, is the number of elements of $\mL$ of rank $i$.
\end{definition}

Computing the characteristic polynomial of a lattice is a standard problem in combinatorics.
For some lattices, this problem is more ``approachable'' than for others.

\begin{definition}
    A \textbf{geometric lattice} is a finite lattice $(\mL, \le)$ which is semimodular and \textbf{atomistic}, i.e., every element $s \in \mL$ is the join of a set of atoms of $\mL$.
\end{definition}

In~\cite{stanley1971modular,stanley1972supersolvable}, Stanley identified a class of geometric lattices whose characteristic polynomial has a precise combinatorial significance. They can be defined as follows.

\begin{definition}
 A geometric lattice $\mL$ is called \textbf{supersolvable} if $\mL$ has  a maximal chain made of modular elements.
\end{definition}

The following result is known as 
Stanley's modular factorization Theorem.
It gives a sufficient condition for the characteristic polynomial of a geometric lattices to have a factor.

\begin{theorem}[see {\cite[Theorem~2]{stanley1971modular}}]
\label{thm:stanley}
    Let $\mL$ be a geometric lattice of rank $n$ and let $t\in\mL$ be a modular element. We have
    \begin{equation*}
         \chi(\mL;\lambda)=\chi([0,t];\lambda)\sum_{\substack{x\in \mL\\x\wedge t=0}}\mu_\mL(x)\lambda^{n-\rho_{\mL}(x)-\rho_{\mL}(t)}.
    \end{equation*}
\end{theorem}

As a corollary of the previous theorem we finally obtain Stanley's decomposition result for the characteristic polynomial of a geometric supersolvable lattice.
The corollary shows that the polynomial splits into linear factors, whose roots have a precise combinatorial significance; see also~\cite[Section~3]{stanley1972supersolvable}.

\begin{corollary}
\label{cor:stanley}
     Let $\mL$ be a supersolvable geometric lattice of rank $n$ and let $0= t_0\cov t_1 \cov \cdots \cov t_n=1$ be a maximal chain of modular elements of $\mL$. We have
     \begin{equation*}
         \chi(\mL;\lambda)=\prod_{i=1}^n(\lambda-|\{a\in\At(\mL)\mid a\leq t_i, \, a\not\leq t_{i-1}\}|).
     \end{equation*}
\end{corollary}

Corollary \ref{cor:stanley} and its generalizations (see e.g. \cite{hallam2015factoring,blass1997mobius}) are standard tools for computing the characteristic polynomial of geometric lattices.
Unfortunately, not all geometric lattices are supersolvable, hence their characteristic polynomial has to be computed using other approaches. This will be the case of \textit{rank-metric lattices}, which we introduce in the next section.

\section{Rank-Metric Lattices}
\label{sec:rml}

In this section we define the main combinatorial objects studied in this paper, namely \textit{rank-metric lattices}.
These can be seen as the $q$-analogue of the higher-weights Dowling lattices, which were introduced in 1971 by Dowling \cite{dowling1971codes} and later investigated and generalized by several authors; see below for further details on this.

We will also illustrate the connection between rank-metric lattices and the problem of computing the number of rank-metric codes of given dimension and minimum distance lower bounded by a given integer $d$. The latter is a central problem in contemporary coding theory.

\begin{notation}
    Throughout this paper, $q$ is a prime power and $n$, $m$ are integers with $n,m\geq 2$. We denote by $\Fq$ the finite field with $q$ elements and by $\F_{q^m}$ its extension of degree $m$.
    We let $\{e_1,\ldots,e_n\}$ be the standard basis of $\Fqn$. With a small abuse of notation we  denote by $\{e_1,\ldots,e_n\}$ also the standard basis of $\Fqmn$. All dimensions in the sequel are computed over $\F_{q^m}$, unless otherwise stated. We denote by $\le$ the inclusion of linear spaces, to emphasize the order structure.
\end{notation}

We start by defining the \textit{rank} of a vector with entries from the field extension $\F_{q^m}$.

\begin{definition}
	The \textbf{rank} (\textbf{weight}) of a vector $v\in\Fqm$ is the dimension over $\Fq$ of the $\F_q$-span of its entries.
	We denote it by $\rk(v)$.
\end{definition}

The rank of a vector is closely related to that of a matrix. To see this, let $\Gamma=\{\gamma_1,\ldots,\gamma_m\}$ be an ordered $\Fq$-basis of $\Fqm$. For a vector $v\in\Fqmn$ we let $\Gamma(v)$ be the $n\times m$ matrix with entries in $\Fq$ defined by 
\begin{equation*}
    v_i=:\sum_{j=1}^m\Gamma(v)_{i,j}\gamma_j \quad \mbox{for all $i$}.
\end{equation*}

  By definition, the map $\Fqmn\longrightarrow\Fq^{n\times m}$ given by $v\longmapsto\Gamma(v)$ is an $\Fq$-isomorphism and we have  $\rk(v)=\rk(\Gamma(v))$ for all bases $\Gamma$ and for all $v \in \F_{q^m}^n$.
  
  The following properties of the vector rank will be needed repeatedly. They easily follow from the connection between vector and matrix rank explained above.
  
  \begin{proposition}
  \label{prop:rk}
            The following hold.
      \begin{enumerate}
          \item $\rk(v)\geq 0$ for all $v\in\Fqmn$, with equality if and only if $v=0$.
          \item $\rk(v)=\rk(\lambda v)$ for all $v\in\Fqmn$ and $\lambda\in\Fqm\setminus\{0\}$.
          \item $\rk(v+w)\leq \rk(v)+\rk(w)$ for all $v,w\in\Fqmn$ \, (triangular inequality).
          \item $\rk(v+w) \ge \rk(v)-\rk(w)$ for all $v,w\in\Fqmn$.
      \end{enumerate}
  \end{proposition}

We are now ready to introduce the class 
\textit{rank-metric lattices},
whose elements are the $\mathbb{F}_{q^m}$-linear subspaces of $\mathbb{F}_{q^m}^n$ having a basis of vectors with rank bounded from above, ordered by inclusion.

\begin{definition}
    Let $i\in \{1, \ldots n\}$. The \textbf{rank-metric lattice} (\textbf{RML} in short) $\mL_i(n,m;q)$ associated with the $4$-tuple $(i,n,m,q)$ is the geometric sublattice of $\mL(\Fqmn)$ whose atoms are the $1$-dimensional $\Fqm$-subspaces of $\Fqmn$ generated by the vectors of rank at most $i$, that is,
    \begin{align*}
        \mL_i(n,m;q):=
        \{\<v_1,\ldots,v_\ell\>_{\Fqm} : \ell \ge 1, \, v_1,\ldots,v_\ell\in T_i(n,m;q)\},
    \end{align*}
    where $T_i(n,m;q):=\{v\in\Fqmn:\rk(v)\leq i\}$. The order is given by the inclusion of $\F_{q^m}$-linear subspaces of $\F_{q^m}^n$. 
\end{definition}

We will need the following observations, which we will use without explicitly referring to them.

\begin{remark}
\label{rem:Li}
    It follows from the definitions that $\mL_i(n,m;q)$ is a geometric lattice of rank~$n$, where the rank function is given by the $\F_{q^m}$-dimension. Furthermore,
     for any $X,Y\in \mL_i(n,m;q)$ we have $X\vee_i Y=X+Y$ (the ordinary sum of linear subspaces) and $X\wedge_i Y=\<x\in X\cap Y:\rk(x)\leq i\>$, where $\vee_i$ and $\wedge_i$ denote the operations of join and meet in $\mL_i(n,m;q)$, respectively.
     By definition we have that
     $\mL_0(n,m;q)$ is the trivial lattice $\{0\}$ and that $\mL_n(n,m;q)=\{X\leq\Fqmn\}$ is the lattice of $\F_{q^m}$-subspaces of $\F_{q^m}^n$.
\end{remark}

As already mentioned, rank-metric lattices can be seen as the $q$-analogues (or also rank-metric analogues) of a well-studied class of lattices, known under the name of  \textit{higher-weight Dowling lattices}. 
These combinatorial geometries were introduced by Dowling in connection with a fundamental problem in coding theory; see~\cite{dowling1971codes,dowling1973q}. 
Their theory and connections with the Critical Problem by Crapo and Rota was further explored by
Zaslavsky~\cite{zaslavsky1987mobius}, Bonin~\cite{bonin1993automorphism,bonin1993modular}, Kung~\cite{kung1996critical}, Brini~\cite{brini1982some}, Games~\cite{games1983packing},
Britz~\cite{britz2005extensions},
and more recently by Ravagnani~\cite{ravagnani2019whitney}. 
The techniques for computing the characteristic polynomials of these lattices for arbitrary parameters are currently unknown.

\begin{definition}
Let $Q$ be a prime power.
The \textbf{Hamming weight} of a vector $v\in\F_Q^n$ is  $\wt^\H(v):=|\{i\in[n]:v_i\neq 0\}|$.
For $i\in\{1,\ldots,n\}$, the \textbf{higher-weight Dowling lattice} (\textbf{HWDL} in short) $\mH_i(n;Q)$ associated to the 3-tuple $(i,n,Q)$ is the geometric sublattice of $\mL(\F_Q^n)$ whose atoms are the $1$-dimensional $\F_Q$-subspaces of $\F_Q^n$ generated by vectors of Hamming weight at most $i$, that is \begin{align*}
		\mH^\H_i(n;Q):=\{\<v_1,\ldots,v_n\>_{\F_Q} : \ell \ge 1, \, v_1,\ldots,v_\ell\in T_i^\H(n;Q)\},
	\end{align*}
	where $T_i^\H(n;Q):=\{v\in\F_Q^n:\wt^\H(v)\leq i\}$.
\end{definition}

Computing the characteristic polynomial (or equivalently the Whitney numbers of the first kind) of higher-weight Dowling lattices
would allow to solve a long-standing open problem in coding theory and finite geometry, namely the famous \textit{Segre's conjecture}, known as the \textit{MDS Conjecture} in the context of coding theory.
More generally, computing the 
characteristic polynomial of higher-weight Dowling lattices
would allow one to give formulas for the number of Hamming-metric codes having given dimension and minimum distance bounded from below by an integer $d$.
The $q$-analogy between 
higher-weight Dowling lattices and rank-metric lattices ``lifts''
to an analogy between the corresponding questions about codes, as we now briefly illustrate.

\begin{definition}
    A (\textbf{rank-metric}) \textbf{code} is an $\F_{q^m}$-linear subspace $C \le \Fqmn$. The \textbf{minimum} (\textbf{rank}) \textbf{distance} of $C$ is $d(C):=\min\{\rk(v):v\in\Fqmn\mid v\neq 0\}$, where the zero code~$\{0\}$ 
    has minimum distance $\min\{n,m\}+1$
    by definition. In this context, $n$ is called the \textbf{length} of $C$. The \textbf{dual} of $C$ is the code $C^\perp:=\{v\in\Fqmn:c\cdot v=0 \textup{ for all } c \in C\}$, where $\cdot$ denotes the standard inner product of vectors.
\end{definition}

The \textit{equivalence} between the problem of counting rank-metric codes with given dimension and minimum distance lower bounded by $d$ and that of computing the Whitney numbers of rank-metric lattices is expressed by 
Theorem~\ref{thm:wjalphak} below
(the general version of the result is~\cite[Theorem~3.1]{ravagnani2019whitney}). In order to state it, we will need the following symbols and the notion of Gaussian binomial coefficients; see~\cite{andrews1998theory}.

\begin{notation} \label{notazvarie}
We let $w_j(i,n,m;q)$ be the $j$-th Whitney number of the second kind of the lattice $\mL_i(n,m;q)$.
We denote by $\alpha_k(i,n,m;q)$ the number of $k$-dimensional rank-metric codes $C \le \F_{q^m}^n$ with $d(C) \ge i+1$.
The $j$-th Whitney number of the first kind of $\mL_i(n,m;q)$ is denoted by $W_j(i,n,m;q)$.
\end{notation}

	\begin{definition}
		Let $a,b$ be integers and let $Q$ be a prime power. The $Q$-binomial coefficients of $a$ and $b$ is
		\begin{equation*}
			\Qbin{a}{b}=
			\begin{cases}
				0 & \textup{ if } b<0 \textup{ or } 0\leq a <b,\\[1ex]
				1 & \textup{ if } b=0 \textup{ and } a\geq 0,\\[1ex]
				\displaystyle\prod_{i=0}^{b-1}\frac{Q^{a-i}-1}{Q^{i+1}-1} & \textup{ if } b>0 \textup{ and } a\geq b.\\[1ex]
			
			\end{cases}
		\end{equation*}
	\end{definition}

We can now state the rank-metric instance of~\cite[Theorem~3.1]{ravagnani2019whitney}.

\begin{theorem}
\label{thm:wjalphak}
The following hold for any $j\in\{0,\ldots,n\}$.
\begin{enumerate}[label=(\arabic*)]
	\item $\displaystyle w_j(i,n,m;q)=\sum_{k=0}^j\alpha_k(i,n,m;q)\qmbin{n-k}{j-k}(-1)^{j-k}q^{m\binom{j-k}{2}}$.
	\item $\displaystyle \alpha_j(i,n,m;q)=\sum_{k=0}^jw_k(i,n,m;q)\qmbin{n-k}{j-k}$.
\end{enumerate}
\end{theorem}

A very natural question is for which values of $j$ we have
$\alpha_j(i,n,m;q)=0$.
A partial answer to this question is given by the following result, which follows from~\cite[Theorem~5.4]{delsarte1978bilinear}.

\begin{theorem}[rank-metric Singleton bound]
	\label{thm:singleton}
	Let $1 \le k \le n$ and $1 \le d \le \min\{n,m\}$ be integers.
	Let $C \le \F_{q^m}^n$ be a code of dimension $k \ge 1$ and minimum distance at least $d$. We have
	    \begin{equation*}
	        mk\leq \max\{n,m\}(\min\{n,m\}-d+1).
	    \end{equation*}
	\end{theorem}
	\begin{proof}
	Let $\Gamma$ be a basis of $\F_{q^m}$ over $\F_q$. Then $\{\Gamma(v) \,:\, v \in C\}$
	is an $\F_q$ linear space of matrices, of the same cardinality as $C$, in which every non-zero matrix has rank at least $d$. Therefore by~\cite[Theorem~5.4]{delsarte1978bilinear} we have that 
	$|C| \le \min\{q^{m(n-d+1)}, q^{n(m-d+1)}\}$. The desired bound follows from the fact that $|C|=q^{mk}$.
\end{proof}

We conclude this section with the following lemma summarizing some well-known properties of Gaussian binomial coefficients. We will need them in various proofs later.

	\begin{lemma}
		\label{lem:bin}
		Let $a,b,c$ be integers, let $Q$ be a prime power, and let $x,y$ be rational numbers. The following hold.
		\begin{enumerate}
			\item $\displaystyle \Qbin{a}{b}\Qbin{b}{c}=\Qbin{a}{c}\Qbin{a-c}{a-b}$.
			\item $\displaystyle \Qbin{a+b}{c}=\sum_{j=0}^cq^{j(b-c+j)}\Qbin{a}{j}\Qbin{b}{c-j}=\sum_{j=0}^cQ^{(c-j)(a-j)}\Qbin{a}{j}\Qbin{b}{c-j}$.
			\item $\displaystyle\sum_{j=0}^c\Qbin{c}{j}(-1)^jQ^{\binom{j}{2}}x^{c-j}y^j=\begin{cases}
				0 & \textup{ if } c=0,\\
				\displaystyle\prod_{j=0}^{c-1}(x-Q^jy) & \textup{ if } c\geq 1.
			\end{cases}$

		\end{enumerate}
	\end{lemma}

\section{Supersolvability}
\label{sec:supersolvable}

As Stanley's modular factorization theorem shows, the characteristic polynomial of supersolvable geometric lattices is particularly easy to compute and completely determined by the combinatorial structure of the lattice's atoms; see Theorem~\ref{thm:stanley}.
It is therefore very natural to ask which rank-metric lattices are supersolvable, a question that we answer completely in this section. 

Continuing the analogy with the Hamming-metric case,
we will first state which higher-weight Dowling lattices are supersolvable. These were fully characterized by Bonin in~\cite{bonin1993modular} and are very few, as the following result shows.

\begin{theorem}[\text{see \cite{bonin1993modular}}]
	\label{thm:dowsupersol}
	The only non-trivial supersolvable higher-weight Dowling lattices are
	$\mH_1(n;q)$, $\mH_2(n;q)$,
	$\mH_{n-1}(n;q)$ for $n\geq 2$,
	and $\mH_n(n;q)$.
\end{theorem}

The rank-metric analogue of 
Theorem~\ref{thm:dowsupersol}, to the proof of which this section is entirely devoted, is the following.

\begin{theorem}
\label{thm:superLi}
    The lattice $\mL_i(n,m;q)$ is supersolvable if and only if one of the following holds:
    \begin{enumerate}
        \item $i\in\{1,n-1,n\}$,
        \item $m\in\{2,\ldots,i\}$.
    \end{enumerate}
\end{theorem}

Note Theorem~\ref{thm:superLi} shows
that the parameter $q$ does not play any role in determining whether or not $\mL_i(n,m;q)$ is supersolvable.

\begin{remark}
The approaches developed in~\cite{bonin1993modular,dowling1973q} extends only in part from the Hamming to the rank metric,
because of technical differences between the two metrics that arise in the proofs.
In fact, 
while the second higher-weight Dowling lattice
$\mH_2(n;q)$ is supersolvable (it is  known as the $q$-analogue of the partition lattice), the second rank-metric lattice~$\mL_2(n,m;q)$ is \textit{not} supersolvable, as we will prove in this section.
Our main result is a complete characterization of the modular elements of $\mL_i(n,m;q)$,
from which Theorem~\ref{thm:superLi}
follows easily. The proof of the characterization is splitted across a number of preliminary results, which we will combine at the end of the section.
\end{remark}

We start with the 
rank-metric analogues of \cite[Lemmas~1.2 and~1.4]{bonin1993modular}. Their proofs extend easily from the Hamming-metric case.

\begin{lemma}
    \label{lem:carmod}
    Let $X \in \mL_i(n,m;q)$. Then $X$ is  modular if and only if $X\wedge_n Y=X\wedge_i Y$ for all $Y\in \mL_i(n,m;q)$. 
\end{lemma}
\begin{proof}
    Recall that, by definition of modularity, $X$ is a modular element of $\mL_i(n,m;q)$ if and only if,
    for all $Y\in \mL_i(n,m;q)$,
    \begin{equation}
    \label{eq:carmod}
        \dim(X\wedge_i Y)=\dim(X)+\dim(Y)-\dim(X\vee_i Y).
    \end{equation}
    Since $X\vee_i Y=X\vee_n Y$, we have that the right hand side of \eqref{eq:carmod} is $X\wedge_n Y$. It follows that~$X$ is modular in $\mL_1(n,m;q)$ if and only if $\dim(X\wedge_i Y)=\dim(X\wedge_n Y)$. The statement now follows from the fact that $X\wedge_i Y\leq X\wedge_n Y$.
\end{proof}

\begin{lemma}
\label{lem:x+y}
    Let $X$ be a modular element of $\mL_i(n,m;q)$ and let $x\in X\setminus T_i(n,m;q)$. If $y\in T_i(n,m;q)$ and $x+y\in T_i(n,m;q)$, then $y\in X$. 
\end{lemma}
\begin{proof}
    We have $\rk(x)\geq i+1$, $\rk(y) \le i$, and $\rk(x+y)\leq i$. Define the lattice element $Y:=\<x+y,y\>_{\Fqm}\in\mL_i(n,m;q)$ and assume toward a contradiction that $y\notin X$.
    Since $y \in Y \setminus X$, we have 
    $X\wedge_n Y=\<x\>_{\Fqm}$.
    Since $x \notin T_i(n,m;q)$, we have $X\wedge_i Y =\{0\}$.
    Therefore $X$ is not modular by Lemma~\ref{lem:carmod}. 
\end{proof}

We will establish Theorem~\ref{thm:superLi} in various steps
and by treating the cases $i=1$, $i=2$, and $i\geq 3$ separately. 
While the general proof strategy is similar to that of~\cite{bonin1993modular}, the technical details are different and heavily rely on the properties of the rank weight.

\begin{proposition}
	\label{prop:L1modular}
	$\mL_1(n,m;q)$ is isomorphic to the lattice of $\F_q$-subspaces of $\F_q^n$. In particular, $\mL_1(n,m;q)$ is a modular lattice.
\end{proposition}
\begin{proof}
We will construct a poset isomorphism
$f: (\mL,\le) \longrightarrow \mL_1(n,m;q)$,
where $(\mL,\le)$ is the lattice of $\F_q$-subspaces of $\F_q^n$ ordered by inclusion.
Define $f(U):=\< U \>_{\Fqm} =U \otimes_{\F_q} \F_{q^m}$ for all $U \in \mL$.
It is easy to see that $f$ is order preserving and injective. Therefore to conclude the proof it suffices to prove that $f$ is surjective.
To see this, let $X \in \mL_1(n,m;q)$  and let $\{v_1,\ldots,v_j\}$ be an $\F_{q^m}$-basis of $X$ made of vectors of rank 1. Write 
$v_\ell =\lambda_\ell w_\ell'$ for $\ell \in \{1,\ldots,j\}$, where $\lambda_\ell \in \F_{q^m} \setminus \{0\}$ and 
$w_\ell \in \F_q^n$.
Then $X=U \otimes_{\F_q} \F_{q^m}=f(U)$, where~$U = \< w_1,\ldots,w_j \>_{\F_q}$, concluding the proof.
\end{proof}

\begin{remark} \label{rem:basis}
The proof of Proposition~\ref{prop:L1modular} shows that the elements of
  $\mL_1(n,m;q)$ are precisely those subspaces of $\F_{q^m}$ having a basis made of vectors with entries in $\F_q$. These are the optimal rank-metric anticodes; see~\cite[Section~4]{ravagnani2016generalized}. 
\end{remark}

The following lemma summarizes some properties of the rank of a vector that will be crucial in our approach. These are mainly inherited from the correspondence between vector rank and matrix rank illustrated in Section~\ref{sec:rml}. They will be applied in several instances throughout this section, without referring to them explicitly, to avoid making the exposition too cumbersome.

\begin{lemma}
\label{lem:sumrkj}
Let $v \in \F_{q^m}^n$ be a non-zero vector. Then $\rk(v)$ is the smallest $r$ for which there exist rank-1 vectors $v_1,\ldots,v_r\in\Fqmn$ with $v=v_1+\ldots+v_r$.
Furthermore, any such rank-1 vectors $v_1,\ldots,v_r$ are linearly independent over $\F_{q^m}$. Finally, for every $S \subseteq \{1,\ldots,r\}$ we have
that $\sum_{j \in S} v_j$ has rank $|S|$.
\end{lemma}
\begin{proof}
The first part of the statement follows from the fact that $v \mapsto \Gamma(v)$ is a rank-preserving $\F_q$-linear isomorphism (for any basis $\Gamma$). 

For the second part, suppose towards a contradiction that the rank-$1$ vectors $v_1, \ldots,v_r$ are not linearly independent over $\Fqm$ and let $\lambda_1,\ldots,\lambda_{r-1}\in\Fqm$ such that (without loss of generality) $v_r=\lambda_1 v_1+\cdots+\lambda_{r-1}v_{r-1}$. We then have
    \begin{equation*}
        v=v_1+\cdots+v_r=(1+\lambda_1)v_1+\cdots+(1+\lambda_{r-1})v_{r-1},
    \end{equation*}
    which implies $\rk(v)\leq r-1$ by the triangular inequality; see Proposition~\ref{prop:rk}. This is a contradiction. 

In order to prove the last part of the statement, let $S \subseteq \{1,\ldots,r\}$ and $w:=\sum_{j \in S} v_j$. Suppose 
     towards a contradiction that $\rk(w)<|S|$. Define the vector $u:=\sum_{j\notin S} v_j$ and observe that $\rk(u)\leq r-|S|$, since it is the sum of $r-|S|$ rank-$1$ vectors. By the triangle inequality we have
    \begin{equation*}
        r=\rk(v)=\rk(w+u)\leq \rk(w)+\rk(u)< |S|+r-|S|=r,
    \end{equation*}
    which is a contradiction.
\end{proof}

The following two lemmas are devoted to the case $i=2$. Note that the proof techniques do not apply to the Hamming-metric case. In fact, the Hamming-metric analogues of the following two lemmas are false; see~\cite{dowling1973q}.

\begin{lemma}
\label{lem:rk3}
    Let $X$ be a modular element of $\mL_2(n,m;q)$. If $X$ contains an element of rank $3$, then $X=\Fqmn$.
\end{lemma}
\begin{proof}
    Suppose that $X$ has an element of rank~$3$, say~$x$. By Lemma~\ref{lem:sumrkj}, there exist $x_1,x_2,x_3\in\Fqmn$ of rank $1$, linearly independent over~$\Fqm$, with $x=x_1+x_2+x_3$. Define $y:=x_1+x_2$ and note that $x-y=x_3$. We have $y,x_3\in T_2(n,m;q)$, which implies that $y,x_3\in X$, by Lemma~\ref{lem:x+y}. Analogously, $x_1,x_2\in X$. If $n=3$, then this immediately shows that $X=\Fqm^3$, because $x_1$, $x_2$ and $x_3$ are linearly independent over~$\F_{q^m}$. 
    
    We assume $n\geq 4$ in the remainder of the proof.  We complete $\{x_1,x_2,x_3\}$ to a basis of~$\Fqmn$ by selecting vectors $x_4,\ldots,x_n\in\Fqmn$ of rank $1$ (observe that is always possible because the rank-$1$ vectors of $\Fqmn$ span $\Fqmn$).  For all $j\in\{1,\ldots,n\}$, let $\lambda_j\in\Fqm \setminus \{0\}$ such that $x_j=\lambda_j y_j$, for some $y_j\in\Fqn \setminus \{0\}$. Note that the $\lambda_j$'s exist because $\rk(x_j)=1$ for all $j\in\{1,\ldots,n\}$. Define the vector $z:=x_1+x_2-\lambda_1y_4$, so that 
    \begin{equation*}
        z=x_1+x_2-\lambda_1y_4=\lambda_1(y_1-y_4)+x_2, \qquad x-z=x_3+\lambda_1 y_4.
    \end{equation*}
    It follows that $z,x-z\in T_2(n,m;q)$, since they are sum of at most two vectors of rank $1$ as $y_1-y_4\in\Fqn$, and therefore $z \in X$ by Lemma~\ref{lem:x+y}. Note that we can re-write $x_4$ as 
    \begin{equation*}
       x_4= \lambda_4y_4=\lambda_1^{-1}\lambda_4(y-z),
    \end{equation*}
    which is a linear combination of  elements of $X$. Therefore $x_4 \in X$ as well. Analogously, one can prove that $x_5,\ldots,x_n\in X$, which in turn implies that $X=\Fqmn$.
\end{proof}

\begin{lemma}
\label{lem:rk4}
    Let $X\in\mL_2(n,m;q)\setminus\{\Fqmn\}$. If $X$ contains an element of rank $4$, then $X$ is not modular.
\end{lemma}
\begin{proof}
     Suppose that $X$ has an element of rank $4$, say $x$. By Lemma~\ref{lem:sumrkj}, there exist $x_1,x_2,x_3,x_4\in\Fqmn$ of rank $1$, linearly independent over $\Fqm$, that sum to $x$. We examine two cases.
     \begin{itemize}
         \item Suppose that $x_j\in X$ for some $j\in\{1,2,3,4\}$. We have that $x-x_j$ is an element of~$X$ of rank~$3$ and therefore $X$ cannot be modular by Lemma~\ref{lem:rk3}.
         \item Suppose $x_1,x_2,x_3,x_4\notin X$ and assume towards a contradiction that $X$ is modular. Define $y:=x_1+x_2$ and note that $y,x-y\in T_2(n,m;q)$. Therefore $y\in X$ by Lemma~\ref{lem:x+y}.
         For all $j\in\{1,2,3,4\}$, take $\lambda_j\in\Fqm\setminus\{0\}$ and $y_j\in\Fqn\setminus\{0\}$ such that $x_j=\lambda_j y_j$, and define 
         \begin{equation*}
            z_1:=x_1+x_2+\lambda_3y_1,\qquad z_2:=x_3+x_4-\lambda_3y_1,\qquad Y:=\<z_1,z_2\>_{\F_{q^m}}.
         \end{equation*}
         Observe that
         \begin{equation*}
             z_1=(\lambda_1+\lambda_3)y_1+x_2, \qquad  z_2=\lambda_3(y_3-y_1)+x_4.
         \end{equation*}
         Since $y_1$, $x_2$, $y_3-y_1$ and $x_4$ all have rank at most 1,
         we have $Y \in \mL_2(n,m;q)$.
         We have $z_1\notin X$, otherwise $z_1-y=\lambda_3 y_1\in X$ and hence $\lambda_1 y_1=x_1\in X$ (since $X$ is a linear space), which is false by assumption. Since $x=z_1+z_2\in Y$ and $Y$ has dimension at most $2$, we have $X\wedge_n Y=\<x\>_{\Fq}$. On the other hand, $x\notin T_2(n,m;q)$ because $\rk(x)=4$, and therefore $X\wedge_2 Y=\<0\>_{\Fq}$. This implies that $X$ is not modular by Lemma \ref{lem:carmod}.  \qedhere
    \end{itemize}
\end{proof}

We next characterize the modular elements of the lattice $\mL_i(n,m;q)$ for $i\geq 3$. Again, the proofs of the following four lemmas heavily use the properties of the rank weight stated in Lemma~\ref{lem:sumrkj}. We start with two results that, technically, only require $i \ge 2$, but that we will need in the case $i \ge 3$ only.

\begin{lemma}
\label{lem:couplediff}
Let $i\geq 2$, $j\in\{1,\ldots,i\}$, and let $X$ be a modular element of $\mL_i(n,m;q)$. Let 
$x_1,\ldots,x_{i+j}\in\Fqmn$ be of rank $1$.
Suppose that 
$x:=x_1+\ldots+x_{i+j} \in X$ and that
$\rk(x) \ge i+1$. Then $x_s-x_p\in X$ for all $s,p \in \{1,\ldots,i+j\}$. 
\end{lemma}
\begin{proof}
    The result is clear if $s=p$. We henceforth assume $s \neq p$, say 
    $(s,p)=(1,2)$
    without loss of generality.  Define $y:=x_1+x_3+\cdots+x_{i+1}$ and $z:=x_2+x_3+\cdots+x_{i+1}$. Observe that $y,z\in X$ by  Lemma~\ref{lem:x+y}, since $x\in X\setminus T_i(n,m;q)$ and  $z,y,x-y,x-z\in T_i(n,m;q)$, as they are sum of at most $i$ vectors of rank $1$ (we are using $j \le i$ here).   Finally, we have that $y-z=x_1-x_2\in X$ because it is sum of elements of $X$.
\end{proof}

\begin{lemma} \label{lll1}
Let $i\geq 2$ and let $X$ be a modular element of $\mL_i(n,m;q)$. If $X$ contains an element of rank strictly bigger than $i$, then it must contain an element of rank either $i+1$ or $i+2$.
\end{lemma}
\begin{proof}
Since $X$ contains an element of rank strictly bigger than $i$ and the sum of two vectors of rank $\le i$ has rank $\le 2i$, $X$ must contain an element $x$ of rank $i+j$, for some $1 \le j \le i$. If $j \in \{1,2\}$ then we are done. Now suppose $j \ge 3$ and write
$x=x_1+ \ldots +x_{i+j}$, where 
$x_1, \ldots x_{i+j} \in \F_{q^m}^n$ have rank 1 and are linearly independent over $\F_{q^m}$.
Let $y:=x_1 + \ldots +x_i$ and observe $y,x-y \in T_i(n,m;q)$ by Lemma~\ref{lem:sumrkj}. Therefore by Lemma~\ref{lem:x+y} we have that $y \in X$.
Let $z_1:=x_{i+1}-x_{i+3}$ and $z_2:=x_{i+2}-x_{i+3}$. By Lemma~\ref{lem:couplediff} we have $z_1,z_2 \in X$, since $X$ is modular. Now observe that
\begin{equation*}
    z:=y+z_1+z_2 = x_1+ \cdots + x_i + x_{i+1} + x_{i+2} -2x_{i+3}
\end{equation*}
has rank $i+1$, $i+2$, or $i+3$,
since by Proposition~\ref{prop:rk} we have
\begin{equation*}
     i+1=\rk(x_1+ \cdots + x_{i+2})-\rk(2x_{i+3})\leq \rk( z)\leq i+3,
\end{equation*}
where the last inequality from the fact that $z$ is a sum of at most $i+3$ vectors of rank $1$. If~$z$ has rank $i+1$ or $i+2$ then we are done. If $z$ has rank $i+3$, then $y+z_1 \in X$ must have rank $i+1$ or $i+2$, since $z_2$ has rank $1$ or $2$ (we are using the triangular inequality again). Therefore in any case we have exhibited 
a vector of rank $i+1$ or $i+2$.
\end{proof}

In the proofs of the following two lemmas we need the assumption $i \ge 3$. For completeness of the exposition we will stress when the assumption is technically used.

\begin{lemma}
\label{lem:i+1} \label{lll2}
   Let $i\geq 3$ and let $X$ be a modular element of $ \mL_i(n,m;q)$. If $X$ contains an element of rank $i+2$, then it contains an element of 
   rank $i+1$.
\end{lemma}
\begin{proof}
    Let $x\in X$ be of rank $i+2$ and write $x:=x_1+\ldots+x_{i+2}$ for some rank-$1$ elements $x_1,\ldots,x_{i+2}\in\Fqmn$ that are linearly independent over $\Fqm$. Define $y:=x_1+\ldots+x_i$ and $z:=x_1+\cdots+x_{i-1}$. Notice that $y,z,x-y,x-z\in T_i(n,m;q)$ by Lemma~\ref{lem:sumrkj} (here we use $i \ge 3$). Therefore, by Lemma \ref{lem:x+y}, $y,z,x-y$ and $x-z$ are elements of $X$. Finally, we have that $x-y+z=x_1+\ldots+x_{i-1}+x_{i+1}+x_{i+2}$ is an element of $X$ of rank $i+1$ by Lemma~\ref{lem:sumrkj}. 
\end{proof}

\begin{lemma}
\label{lem:fullspace} \label{lll3}
    Let $i\geq 3$ and let $X$ be a modular element of $\mL_i(n,m;q)$. If $X$ contains an element of rank $i+1$, then $X=\Fqmn$.  
\end{lemma}
\begin{proof}
    Suppose that $X$ contains an element of rank $i+1$, say $x$. Write $x=x_1+\ldots+x_{i+1}$ for some $x_1,\ldots,x_{i+1}\in\Fqmn$ of rank $1$ and linearly independent over~$\Fqm$. 
    Note that for all
    $j\in\{1,\ldots, i+1\}$
    we have $x-x_j\in T_i(n,m;q)$.
    Therefore by Lemma~\ref{lem:x+y} we have $x_j\in X$ for all $j$. We now complete $\{ x_1, \ldots, x_{i+1}\}$ to a basis
    of $\F_{q^m}^n$ by selecting vectors 
    $x_{i+2},\ldots,x_{n}\in\Fqmn$ of rank $1$ (this is always possible because the rank-$1$ vectors of $\Fqmn$ span~$\Fqmn$). Thus $\<x_1,\ldots,x_{i+1},x_{i+2},\ldots,x_n\>_{\Fqm}=\Fqmn$.  For $s\in\{i+2,\ldots,n\}$, define the vector $z_s:=x_1+x_2-x_s$ and notice that $z_s,x-z_s\in T_i(n,m;q)$ by construction (here is where the assumption $i \ge 3$ is technically needed). Lemma \ref{lem:x+y} implies that $z_s\in X$ and therefore $x_s\in X$, for all $s\in\{i+2,\ldots,n\}$, since $x_1,x_2\in X$ by the first part of the proof and~$X$ is a linear space. It follows that $X$ contains a basis for $\Fqmn$ and therefore $X=\Fqmn$, as claimed.
\end{proof}

We can finally characterize the 
modular elements of any lattice
$\mL_i(n,m;q)$ for $2 \le i \le n$.
The following is the analogue of~\cite[Theorem~1.3]{bonin1993modular}.
Note that in the rank-metric context (and in sharp contrast to the Hamming-metric case) it also holds for $i=2$.

\begin{theorem}
\label{thm:modLi}
    For all $i\geq 2$, the modular elements of $\mL_i(n,m;q)$ are the elements of the set $\{X\leq\Fqmn:X\subseteq T_i(n,m;q)\}\cup\{\Fqmn\}$. 
\end{theorem}
\begin{proof}
    First of all we note that if $X\subseteq T_i(n,m;q)$, then $X\wedge_i Y\in \mL_i(n,m;q)$ and therefore $X\wedge_i Y=X\wedge_n Y$, for all $Y\in \mL_i(n,m;q)$. It follows that $X$ is modular in $\mL_i(n,m;q)$ by Lemma~\ref{lem:carmod}. Moreover, $\Fqmn$ is trivially modular. 
    
    In order to prove the other implication, suppose that
    $X$ is a modular element of $\mL_i(n,m;q)$ that contains a vector of rank strictly greater than~$i$.
    By combining Lemmas~\ref{lem:rk3} and \ref{lem:rk4} in the case $i=2$ and Lemmas~\ref{lll1}, \ref{lll2} and \ref{lll3} in the case $i \ge 3$, we conclude that $X=\F_{q^m}^n$. This establishes the theorem.
\end{proof}

We are finally ready to establish Theorem \ref{thm:superLi}.

\begin{proof}[Proof of Theorem \ref{thm:superLi}]
    The supersolvability of the lattice $\mL_1(n,m;q)$ immediately follows from the fact that 
    $\mL_1(n,m;q)$ is a modular lattice
    by Proposition \ref{prop:L1modular}. We assume $i \ge 2$ in the remainder of the proof. Let $\{e_1,\ldots,e_n\}$ be the canonical basis of
    $\Fqmn$. Consider the following maximal chain in $\mL_i(n,m;q)$:
    \begin{equation}
    \label{eq:maxchain}
        \<0\>_{\Fqm}\cov \<e_1\>_{\Fqm}\cov \<e_1,e_2\>_{\Fqm}\cov \cdots \cov \<e_1,e_2,\ldots,e_n\>_{\Fqm}.
    \end{equation}
    For $1 \le k \le n$, all the elements of  $\<e_1,e_2,\ldots,e_k\>_{\Fqm}$ have rank at most $k$ by the triangular inequality.
    Moreover, Theorem~\ref{thm:modLi} implies that if $i\in\{n-1,n\}$ or $m\in\{2,\ldots,i\}$, then all the element of the chain in \eqref{eq:maxchain} are modular. This shows that, by definition, $\mL_i(n,m;q)$
    is supersolvable in those cases.
    
    We conclude the proof by showing that $\mL_i(n,m;q)$ cannot be supersolvable if $m\geq i+1$ and $i\in\{2,\ldots,n-2\}$. We will do this by showing a much stronger fact: All the elements of~$\mL_i(n,m;q)$ of rank (i.e., dimension) $i+1$ cannot be modular. 
   To see this, let $X\in \mL_i(n,m;q)$ have dimension $i+1$ and let $\smash{G\leq\Fqm^{(i+1)\times n}}$ be the full-rank matrix in reduced row-echelon form
    whose rows are a basis for $X$. 
    Let $\Gamma:=\{\gamma_1,\ldots,\gamma_{i+1}\}$ be $\F_q$-linearly independent elements of $\Fqm$ (they exist because $m \ge i+1$).
    One can easily check that $\sum_{s=1}^{i+1}\gamma_sG_s$ has rank at least $i+1$. Therefore $X$ cannot be modular by Theorem~\ref{thm:modLi}.
\end{proof}

Table~\ref{thetable} illustrates Theorem \ref{thm:superLi} visually. For a fixed
$n\geq 3$, it shows the ``supersolvability region'' of $\mL_i(n,m;q)$.

\renewcommand\arraystretch{1.5}
\begin{table}
\centering
\begin{tabular}[c]{{|>{\centering\arraybackslash}p{0.5cm}}*8{|>{\centering\arraybackslash}p{0.8cm}}|}
\cline{3-9}
\multicolumn{2}{c|}{}&\multicolumn{7}{c|}{$m$}\\ \cline{3-9}
\multicolumn{2}{c|}{}& $2$ & $3$ & $4$ & $\cdots$ & $j$ & $j+1$ & $\cdots$ \\\hline
\multirow{8}{*}{$i$}&$1$ & \cb & \cb & \cb & \cb & \cb & \cb & \cb \\\hhline{~|*8{-}}
& $2$ & \cb & & & & & & \\\hhline{~|*8{-}}
&$3$ & \cb & \cb &  &  &  &  &  \\\hhline{~|*8{-}}
& $\vdots$ & \cb & \cb & \cb  & & & &  \\\hhline{~|*8{-}}
&$j$ & \cb & \cb & \cb & \cb &\cb &  & \\\noalign{\vskip-0.5pt}\hhline{~|*8{-}}
& $\vdots$ & \cb & \cb & \cb & \cb  &\cb & \cb  &  \\\noalign{\vskip-1pt}\hhline{~|*8{-}}
&$n-1$ &\cb & \cb & \cb & \cb & \cb & \cb & \cb \\\noalign{\vskip-1pt}\hhline{~|*8{-}}
&$n$ &\cb & \cb & \cb & \cb & \cb & \cb & \cb \\\hline
\end{tabular}

\caption{The ``region'' of supersolvable rank-metric lattices for $n\geq 3$. \label{thetable}}
\end{table}
\renewcommand\arraystretch{1}

\section{The Characteristic Polynomial of $\mL_i(n,m;q)$}
\label{secPAP:5}

In this section we derive explicit formulas for the Whitney numbers of the first and second kind of 
some rank-metric lattices,
including the supersolvable ones 
we characterized 
in Theorem~\ref{thm:superLi}. We also compute the \textit{critical exponent} (defined later) of rank-metric
lattices for some parameter sets.
We start with the case $i=1$.

\begin{proposition}
\label{prop:whitneyL1}
    The following hold.
    \begin{enumerate}[label=(\arabic*)]
        \item $w_j(1,n,m;q)=(-1)^jq^{\binom{j}{2}}\qbin{n}{j}$ for all $j\in\{0,\ldots,n\}$.
        \item $W_j(1,n,m;q)=\qbin{n}{j}$ for all $j\in\{0,\ldots,n\}$.
        \item\label{item3:wL1} $\displaystyle  \chi(\mL_1(n,m;q);\lambda)=\sum_{j=0}^n(-1)^jq^{\binom{j}{2}}\qbin{n}{j}\lambda^{n-j}=\prod_{j=0}^{n-1}(\lambda-q^j).$
    \end{enumerate}
\end{proposition}
\begin{proof}
    The first two formulas follow from Proposition~\ref{prop:L1modular} and the definitions of Whitney numbers of the first and second kind. The third formula  follows from the definition of characteristic polynomial and Lemma~\ref{lem:bin}.
\end{proof}

We now treat the case $i=n$.
The following result easily follows from the definitions
and its proof is therefore omitted.

\begin{proposition}
\label{prop:whitneyLn}
    The following hold.
    \begin{enumerate}[label=(\arabic*)]
        \item $w_j(n,n,m;q)=(-1)^jq^{m\binom{j}{2}}\qmbin{n}{j}$ for all $j\in\{0,\ldots,n\}$.
        \item $W_j(n,n,m;q)=\qmbin{n}{j}$ for all $j\in\{0,\ldots,n\}$.
        \item $\displaystyle  \chi(\mL_n(n,m;q);\lambda)=\sum_{j=0}^n(-1)^jq^{m\binom{j}{2}}\qmbin{n}{j}\lambda^{n-j}=\prod_{j=0}^{n-1}(\lambda-q^{mj}).$
    \end{enumerate}
\end{proposition}

It is interesting to observe the very close ``structural'' analogy
between Propositions~\ref{prop:whitneyL1} and~\ref{prop:whitneyLn}, where the only difference is the presence of the factor $m$.

Since the rank of a vector $v \in \F_{q^m}^n$ cannot exceed $m$, the next result follows from 
Proposition \ref{prop:whitneyLn}.

\begin{corollary}
Suppose that $m\in\{2,\ldots,i\}$. The following hold.
    \begin{enumerate}[label=(\arabic*)]
        \item $w_j(i,n,m;q)=(-1)^jq^{m\binom{j}{2}}\qmbin{n}{j}$ for all $j\in\{0,\ldots,n\}$.
        \item $W_j(i,n,m;q)=\qmbin{n}{j}$ for all $j\in\{0,\ldots,n\}$.
        \item $\displaystyle  \chi(\mL_i(n,m;q);\lambda)=\sum_{j=0}^n(-1)^jq^{m\binom{j}{2}}\qmbin{n}{j}\lambda^{n-j}=\prod_{j=0}^{n-1}(\lambda-q^{mj}).$
    \end{enumerate}
\end{corollary}

In the next result we study the case $i=n-1$ under the assumption that $m \ge n$. We will use Stanley's modular factorization theorem.

\begin{theorem}
\label{thm:polLn-1}
    If $n\leq m$, then 
    \begin{equation*}
        \chi(\mL_{n-1}(n,m;q);\lambda)=\left(\lambda-q^{(n-1)m}+\prod_{s=1}^{n-1}(q^m-q^s)\right)\prod_{j=0}^{n-2}(\lambda-q^{mj}).
    \end{equation*}
\end{theorem}
\begin{proof} By Corollary~\ref{cor:stanley} and the proof of Theorem \ref{thm:superLi}, we have 
    \begin{multline} \label{ff}
        \chi(\mL_{n-1}(n,m;q);\lambda)\\=\prod_{j=1}^{n}\left(\lambda-\left|\left\{A\in\textup{At}(\mL_i(n,m;q)) : A\leq \<e_1,\ldots,e_{j}\>_{\Fqm}, \, A\not\leq \<e_1,\ldots,e_{j-1}\>_{\Fqm}\right\}\right|\right),
    \end{multline}
    were we set the span over $\Fqm$ of the empty set to be $\{0\}$. Therefore it remains to compute the quantities 
    \begin{equation*}
        a_j:=\left|\left\{A\in\textup{At}(\mL_{n-1}(n,m;q)):A\leq \<e_1,\ldots,e_j\>_{\Fqm},A\not\leq \<e_1,\ldots,e_{j-1}\>_{\Fqm}\right\}\right|
    \end{equation*}
    for $j\in\{1,\ldots,n\}$. Observe that the atoms of $\mL_i(n,m,q)$ below $\<e_1,\ldots,e_{j}\>_{\Fqm}$ but not below $\<e_1,\ldots,e_{j-1}\>_{\Fqm}$ are the $1$-dimensional subspace of $\Fqmn$ generated by a vector of the form $(\star,\cdots,\star,1,0,\cdots,0)$, where the $1$ is in position $j$ and the $\star$'s are arbitrary elements of $\F_{q^m}$. It follows that $a_j=q^{(j-1)m}$, for all  $j\in\{1,\ldots,n-1\}$. Moreover,     \begin{equation*}
        a_n=q^{(n-1)m}- \left|\left\{\<x\>_{\Fqm}\leq\Fqmn:\rk(x)=n\right\}\right|=q^{(n-1)m}-\prod_{s=1}^{n-1}(q^m-q^s).
    \end{equation*}
    Substituting these expressions into~\eqref{ff}
    gives the desired result.
\end{proof}

We now extract from
Theorem~\ref{thm:polLn-1} some more explicit information about the Whitney numbers of $\mL_{n-1}(n,m;q)$.

\begin{corollary}
    The following hold for $n\leq m$. 
    \begin{enumerate}[label=(\arabic*)]
    \setlength{\itemsep}{3pt}
        \item\label{item1:wLn-1} $w_0(n-1,n,m;q)=1$.
        \item\label{item2:wLn-1} $\displaystyle w_j(n-1,n,m;q)=(-1)^j\left(q^{m\binom{j}{2}}\qmbin{n}{j}-q^{m\binom{j-1}{2}}\qmbin{n-1}{j-1}\prod_{s=1}^{n-1}(q^m-q^s)\right)$ for all $j\in\{1,\ldots,n\}$.
        \item\label{item3:wLn-1} $W_0(n-1,n,m;q)=1$.
        \item\label{item4:wLn-1} $\displaystyle W_1(n-1,n,m;q)=\qmbin{n}{1}-\prod_{s=1}^{n-1}(q^m-q^s)$.
        \item\label{item5:wLn-1} $W_j(n-1,n,m;q)=\qmbin{n}{j}$ for all $j\in\{2,\ldots,n\}$.
    \end{enumerate}
\end{corollary}
\begin{proof}
    Parts \ref{item1:wLn-1} and  \ref{item3:wLn-1} follow from the definitions. Part \ref{item2:wLn-1} follows from Theorem \ref{thm:polLn-1}, since by Lemma~\ref{lem:bin}
    we have
    \begin{multline*}
        \chi(\mL_{n-1}(n,m;q);\lambda)\\
        \begin{aligned}
        &=\prod_{j=0}^{n-1}(\lambda-q^{mj})+\prod_{s=1}^{n-1}(q^m-q^s)\prod_{j=0}^{n-2}(\lambda-q^{mj})\\
        &=\sum_{j=0}^n(-1)^jq^{m\binom{j}{2}}\qmbin{n}{j}\lambda^{n-j}+\prod_{s=1}^{n-1}(q^m-q^s)\sum_{j=1}^n(-1)^{j-1}q^{m\binom{j-1}{2}}\qmbin{n-1}{j-1}\lambda^{n-j}.
        \end{aligned}
    \end{multline*}
    By the definition of Whitney numbers of the second kind we have
    \begin{equation*}
        W_1(n-1,n,m;q)=\qmbin{n}{1}-\left|\left\{\<x\>_{\Fqm}\leq\Fqmn:\rk(x)=n\right\}\right|,
    \end{equation*}
    which implies \ref{item4:wLn-1}. Finally, \ref{item5:wLn-1} follows from the observation that any subspace of $\Fqmn$ of dimension 
    $j \ge 2$ has a basis of vectors of rank $\leq n-1$. 
    To see this, let $X\leq\Fqmn$ be of dimension $j\in\{2,\ldots,n\}$ and let $\smash{G\in\Fqm^{j\times n}}$ be the matrix in reduced row-echelon form whose rows form a basis of $X$. Every row of $X$ has at least $j-1$ zeroes. Therefore its rank is at most $n-j+1 \ge n-1$.
\end{proof}

The following result is the $q$-analogue of \cite[Theorem~3.3]{bonin1993modular}. It computes some of the roots of the characteristic polynomial of rank-metric lattices, which is a standard question in lattice theory.
    
\begin{proposition}
\label{prop:roots}
For all $i \in \{1,\ldots,n\}$, the characteristic polynomial of $\mL_i(n,m;q)$ has the integers $1,q^m,\ldots,q^{(i-1)m}$ among its roots.
\end{proposition}
\begin{proof}
    Let $\smash{U_i:=\<e_1,\ldots,e_i\>_{\F_{q^m}}}$. All the elements of $U_i$ have rank less or equal than $i$. Therefore~$U_i$ is a modular element of $\mL_i(n,m;q)$ by Theorem~\ref{thm:modLi}. As an application of Theorem~\ref{thm:stanley} we obtain
    \begin{equation} \label{fff}
        \chi(\mL_i(n,m;q);\lambda)=\chi([0,U_i];\lambda)\sum_{\substack{X\in \mL_i(n,m;q)\\X\wedge U_i=0}}\mu(0,X) \, \lambda^{n-i-\dim(X)},
    \end{equation}
    where $\mu$ is the M\"obius function of $\mL_i(n,m;q)$.
    The desired result now follows from the fact that the interval $[0,U_i]$ is the lattice of subspaces of~$U_i$, whose characteristic polynomial is 
    \begin{equation*}
        \chi([0,U_i];\lambda)=\prod_{s=0}^{i-1}\left(\lambda-q^{sm}\right);
    \end{equation*}
see e.g.~\cite{stanley2011enumerative}.
    Combining this with~\eqref{fff} concludes the proof.
\end{proof}

 The following example shows that converse of Corollary~\ref{cor:stanley} does not hold in general. Recall that, by Proposition~\ref{prop:L1modular}, the rank-metric lattice $\mL_2(n,m;q)$ is \textit{not} supersolvable.

\begin{example} 
    Let $q=2$, $n=4$, and $m=3$. We have $\alpha_0(2,4,3;2)=1$. Observe that the maximum rank of an element of $\F_{2^3}^4$ is $3$ and that  
    \begin{equation*}
        \alpha_1(2,4,3;2)=\left|\left\{\<v\>_{\F_{2^3}}:v \in\F_{2^3}^4, \,  \rk(v)=3\right\}\right|
          =\qbinomial{4}{3}{2}(2^3-2)(2^3-2^2)=360.
    \end{equation*}
We also have
    \begin{equation*}
        \alpha_2(2,4,3;2)=\left|\left\{C\leq\F_{2^3}^4:\dim(C)=2, \, d(C)\geq 3\right\}\right|.
    \end{equation*}
    On the other hand, by the rank-metric Singleton bound of Theorem~\ref{thm:singleton} we have $6\leq 4(3-d+1)$ for any $2$-dimensional code of $\F_{2^3}^4$ of minimum distance $d$. This implies $d\leq 2$ and therefore $\alpha_2(2,4,3;2)=0$.  A similar argument using again the rank-metric Singleton bound of Theorem~\ref{thm:singleton} shows that $\alpha_3(2,4,3;2)=
    \alpha_4(2,4,3;2)=0$. Finally, by applying  Theorem~\ref{thm:wjalphak} we obtain 
    \begin{equation*}
        (w_j(2,4,3;2):j \in \{0,\ldots,4\})=(1, -225, 11680, -89280, 77824).
    \end{equation*}
    Therefore, by definition, we have $\chi(\mL_2(4,3;2);\lambda)=\lambda^4 - 225\lambda^3 + 11680\lambda^2 - 89280\lambda + 77824$, whose factorization is $(\lambda-1)(\lambda-8)(\lambda-64)(\lambda-152)$. However, $\mL_2(4,3;2)$ is not supersolvable by Theorem~\ref{thm:superLi}.
\end{example}

\begin{remark}
    The linear factor $\lambda-64=\lambda-2^{2 \cdot 3}$ in the previous example
    might suggest that the characteristic polynomial of $\mL_2(4,3;2)$ can be computed 
    as in the proof of Theorem~\ref{thm:polLn-1} via the chain
    $$0 \cov \langle e_1 \rangle_{\F_{2^3}} \cov \langle e_1,e_2 \rangle_{\F_{2^3}} \cov 
    \langle e_1,e_2,e_3 \rangle_{\F_{2^3}} \cov
    \F_{2^3}^4.$$
    However, the proof of Lemma~\ref{lem:rk3} shows that
    $\langle e_1,e_2,e_3 \rangle_{\F_{2^3}}$ is not modular.
    To see this in a more direct way, 
    let $1,\alpha,\beta \in \F_{2^3}$ be linearly independent over $\F_2$. Define $x=e_1+\alpha e_2 + \beta e_3$ and $z=e_1+e_4+\alpha e_2$.
    Let $U = \langle z,x+z \rangle_{\F_{2^3}}$.
    It is easy to check that $X\in \mL_2(4,3;2)$. Moreover, $\langle e_1,e_2,e_3 \rangle_{\F_{2^3}} \cap U = \langle x \rangle$, because
$z \in U \setminus \langle e_1,e_2,e_3 \rangle_{\F_{2^3}}$.
This contradicts the definition of modularity for $\langle e_1,e_2,e_3 \rangle_{\F_{2^3}}$.
\end{remark}

We now turn to the \textit{critical exponent}
of rank-metric lattices, which is one of the most important invariant of a combinatorial geometry. It was defined by Crapo and Rota in~\cite{crapo1970foundations} and it measures the largest dimension of a linear space that \textit{distinguishes}, i.e., avoids, a given collection of projective points. 
Computing the critical exponent of a combinatorial geometry is a hard problem in general, known under the name of \textit{Critical Problem}; see~\cite{crapo1970foundations}. 

In this paper, we give the definition of critical exponent directly in the context of rank-metric lattices.

\begin{definition}
\label{def:crit}
    The \textbf{critical exponent} of $\mL_i(n,m;q)$ is $$\crit(\mL_i(n,m;q))=n-\max\{k\in\Z_{\geq 0}:\alpha_k(i,n,m;q) \neq 0\},$$
    where $\alpha_k(i,n,m;q)$ is defined in Notation~\ref{notazvarie}.
\end{definition}

In the next result we explicitly compute the value of the critical exponent for some parameter sets.

\begin{theorem}
\label{thm:criticalexp}
     \begin{enumerate}[label=(\arabic*)]
        \item If $n \le m$, then $\crit(\mL_i(n,m;q))=i$.
        \item If $n >m$, then $\crit(\mL_i(n,m;q))\geq ni/m$.
        Furthermore, $\crit(\mL_i(n,m;q))=ni/m$ if  $m\mid ni$ and one of the following holds:
        \begin{enumerate}
            \item $n=tm$ and $i=m-k$ for some nonnegative integers $t, k$;
            \item $n=mk/2$ for some nonnegative integer $k$, $m$ even, and $i\in\{2,m-2\}$;
            \item $n=8$, $m=6$, $i=3$ and $q=2^h$ with $h\equiv 1 \mod 6$;
            \item $n=8$, $m=4$, $i=3$ and $q=2^h$ for some odd nonnegative integer $h$.
        \end{enumerate}
    \end{enumerate}
\end{theorem}
\begin{proof}
    We start observing that, by definition, we have
    \begin{multline}
    \label{eq:max}
        \max\{k\in\Z_{>0}:\alpha_k(i,n,m;q) \neq 0\}\\
        \begin{aligned}
        &=\max\{k\in\Z_{>0}:\exists\, C\leq\Fqmn \textup{ s.t. } \dim(C)=k \textup{ and } d(C)\geq i+1\} \nonumber \\
        &\leq \frac{\max\{n,m\}(\min\{n,m\}-i)}{m},
        \end{aligned}
    \end{multline}
    where the latter inequality follows from the rank-metric Singleton bound; see Theorem~\ref{thm:singleton}.
    Therefore, by the definition of critical exponent, we have
    \begin{equation} \label{critb}
    \crit(\mL_i(n,m;q)) \geq n-\frac{\max\{n,m\}(\min\{n,m\}-i)}{m},
    \end{equation}
    where we have equality in~\eqref{critb} if and only if there exists a code $C \leq \F_{q^m}^n$
of dimension    $\max\{n,m\}(\min\{n,m\}-i)/{m}$
and minimum distance at least $i+1$.
Again by the rank-metric Singleton bound (Theorem~\ref{thm:singleton}), this happens if and only if there exists a code $C \leq \F_{q^m}^n$
of dimension    $\max\{n,m\}(\min\{n,m\}-i)/{m}$
and minimum distance exactly $i+1$.
In the remainder of the proof we describe some parameter sets for which the existence of such a code has been established.

If $m \ge n$ then 
the existence of a code $C \le \F_{q^m}^n$
of dimension $n-i$ and minimum distance $i+1$
has been established by Delsarte~\cite{delsarte1978bilinear} for all $i$. If
$m<n$, the existence is only known for some parameters, which have been conveniently collected in~\cite[Remark~2.3]{marino2022evasive} and to which we refer (see also
\cite{bartoli2018maximum,csajbok2017maximum,bartoli2022new}).
These correspond precisely to the parameters 
listed in the statement.
\end{proof}

By combining Theorem~\ref{thm:criticalexp}
with \cite[Corollary~3.4]{ravagnani2019whitney}, we obtain a recursion for the Whitney numbers of the first kind of $\mL_i(n,m;q)$ in some parameter ranges.

\begin{proposition}
    Let $j\in\{0,\ldots,n\}$. We have 
    \begin{equation}
    \label{eq:wjcritical}
        w_j(i,n,m;q)=-\sum_{s=0}^{j-1} w_s(i,n,m;q)\qmbin{n-s}{j-s}
    \end{equation}
    if $n\leq m$ and $n-i<j$, or if 
    $n>m$ and $n(m-i)/m<j$.
\end{proposition}

\section{On the Whitney Numbers of $\mL_2(4,4;q)$}
\label{sec:specific_chi}

In this section we concentrate on the ``smallest'' non-supersolvable rank-metric lattices whose characteristic polynomials we cannot compute by applying the results of Section~\ref{secPAP:5}, namely, lattices of the form~$\mL_2(4,4;q)$. We give explicit formulas for their characteristic polynomials, \textit{under the assumption} that their second Whitney numbers of the first kind are a polynomial in $q$. This technical assumption will allow us to follow a computational approach based on rank-metric codes, Theorem~\ref{thm:wjalphak}, and Lagrange interpolation.

\begin{notation}
\label{not:Mq}
In this section, $n=m=4$. We denote by $M(q)$ the number of $2$-dimensional MRD codes $C \le \F_{q^4}^4$ to simplify the notation.
 \end{notation}

We start by observing the following.

\begin{proposition}
\label{prop:fandstuff}
	We have:
	\begin{enumerate}[label=(\arabic*)]
		\item $\displaystyle w_2(2,4,4;q)=q^4\bbin{4}{2}{4}-\bbin{3}{1}{4} \, \sum_{j=3}^4\qbin{4}{j} \, \prod_{s=1}^{j-1}(q^4-q^s)+M(q)$;
		\item $\displaystyle W_2(2,4,4;q)=\bbin{4}{2}{4}-M(q)$.
	\end{enumerate}
	Moreover, $w_2(2,4,4;q)$ is a polynomial in $q$ if and only if $M(q)$ is a polynomial in $q$.
\end{proposition}
\begin{proof}
	By Theorem~\ref{thm:wjalphak}, we have 
	\begin{equation} \label{hei}
		w_2(2,4,4;q)=\sum_{k=0}^2\alpha_k(2,4,4;q)\qbinomial{4-k}{2-k}{q^4}(-1)^{2-k} \, q^{4\binom{2-k}{2}}.
	\end{equation}
	Furthermore:
	\begin{itemize}
		\item $\alpha_0(2,4,4;q)=1$;
		\item $\displaystyle\alpha_1(2,4,4;q)=\sum_{j=3}^4\qbin{4}{j} \, \prod_{s=1}^{j-1}(q^4-q^s)$;
		\item $\alpha_2(2,4,4;q)=M(q)$.
	\end{itemize}
	These expressions combined with~\eqref{hei} give the
	first part of the statement. For the second part, note that by Theorem~\ref{thm:singleton} we have $d(C)\leq 3$ for all $C\leq\F_{q^4}^4$ of dimension 2. Therefore
	\begin{equation*}
		W_2(2,4,4;q)=\bbin{4}{2}{4}-|\{C\leq\F_{q^4}^4:d(C)> 2\}|=\bbin{4}{2}{4}-M(q).
	\end{equation*}
	The last part of the statement follows from Theorem~\ref{thm:wjalphak} and the fact that $\alpha_3(2,4,4;q)=\alpha_4(2,4,4;q)=0$ by Theorem~\ref{thm:singleton}.
\end{proof}

The following observation will be crucial in our approach.

\begin{lemma}
\label{lem:constt}
Let $u,v \in \F_{2^4}^4$ be vectors of rank 4. The number of 2-dimensional MRD codes $C \le \F_{2^4}^4$ containing $u$ is the same as the number of 2-dimensional MRD codes $C \le \F_{2^4}^4$ containing $v$.
\end{lemma}

\begin{proof}
    Since the entries of $u$ and $v$ both span $\F_{q^4}$ over $\F_2$, for all $j \in \{1,\ldots,i\}$ there exist 
    $\smash{a_{1,j}, \ldots a_{4,j} \in \F_2}$ with $\smash{v_j=\sum_{i=1}^4 u_ia_{i,j}}$. Let $A$ be the matrix whose $(i,j)$ entry is~$a_{i,j}$. Then~$A$ has size $4 \times 4$ and invertible.
    Moreover, $v=u \cdot A$. Therefore the mapping 
    $\F_{q^4}^4 \to \F_{q^4}^4$ given by $x \mapsto x \cdot A$
    is rank-preserving and thus induces a bijection between the
    2-dimensional MRD codes $C \le \F_{q^4}^4$ containing $u$ and the 2-dimensional MRD codes $C \le \F_{q^4}^4$ containing~$v$. This establishes the lemma.
\end{proof}

\begin{notation}
We let $\smash{\hat M(q)}$ be the quantity 
defined by the previous lemma, 
that is,~$\smash{\hat M(q)}$ is the number of $2$-dimensional MRD codes in $\smash{\F_{q^4}^4}$ containing a \textit{given} vector of rank $4$.
\end{notation}

The quantities $M(q)$ and $\hat{M}(q)$ are closely connected to each other, as the following result shows.

\begin{proposition}
\label{prop:M'}
    We have
    \begin{equation*}
       \hat M(q)=\frac{M(q)(q^3-q^2-q-1)}{q^5(q-1)^3(q+1)(q^2+q+1)}.
    \end{equation*}
\end{proposition}
\begin{proof}
    We count the elements of the set 
    \begin{equation}
    \label{eq:(C,x)}
        \Sigma:=\{(C,v) : C \le \F_{q^4}^4 \textup{ is MRD}, \, \dim(C)=2, \, v \in \F_{q^4}^4, \, \rk(v)=2\}
    \end{equation}
    in two ways. 
    Since the number of rank-4 vectors in 
    $\F_{q^4}^4$ is 
    $(q^4-1)(q^4-q)(q^4-q^2)(q^4-q^3)$, by Lemma~\ref{lem:constt} we have
    \begin{equation}
    \label{eq:Mhat}
        |\Sigma|= \hat M(q)(q^4-1)(q^4-q)(q^4-q^2)(q^4-q^3).
    \end{equation}
    Now fix a $2$-dimensional MRD code $C$ in $\smash{\F_{2^4}^4}$
    and recall for every such $C$ the number of rank-4 vectors $v \in C$ is the same and is computed, for example,
    in~\cite[Corollary~28]{de2018weight}. Therefore,
    \begin{equation}
    \label{eq:M}
    \begin{aligned}
    |\Sigma| &=M(q)\left(q^8-1-\qbin{4}{1}(q^4-1)\right)\\
    &=M(q) \, q \, (q-1) \, (q+1) \, (q^2+1)\, (q^3-q^2-q-1).
    \end{aligned}
    \end{equation}
The statement follows from comparing~\eqref{eq:Mhat} with~\eqref{eq:M}.
\end{proof}

We now start working under the assumption that
$w_2(2,4,3;q)$ is a polynomial in $q$. By Proposition~\ref{prop:fandstuff}, this is equivalent to assuming that $M(q)$ is a polynomial in $q$.

The values of $\hat{M}(q)$ for some small values of $q$ (we managed up to $q=9$) can be computed using algebra software combined with an argument that we outline very briefly.
We let $\alpha$ be a primitive element of $\F_{q^4}$
and $v:=(1,\alpha,\alpha^2,\alpha^3)$ be a vector of rank $4$.
Then $\smash{\hat{M}(q)}$ is the number of 2-dimensional MRD codes
$\smash{C \le \F_{q^4}^4}$ containing $v$.
We consider the code whose generator matrix is
\begin{equation*}
    \begin{pmatrix}
    1 & \alpha & \alpha^2 & \alpha^3 \\ 0 & 1 & A & B
    \end{pmatrix}
\end{equation*}
for some $A,B\in\F_{q^4}^4$ such that $\{1,A,B\}$ is a set of elements that are linearly independent over $\Fq$. Observe that the matrices 
\begin{equation*}
    \begin{pmatrix}
    1 & \alpha & \alpha^2 & \alpha^3 \\ 0 & 1 & A & B
\end{pmatrix}, \qquad
\begin{pmatrix}
    1 & \alpha & \alpha^2 & \alpha^3 \\ 0 & 1 & A' & B'
\end{pmatrix}
\end{equation*}
generate the same code if and only if $A=A'$
and $B=B'$. Therefore to compute $\hat{M}(q)$ it suffices 
to count all the pairs $(A,B)$ such that, for all non-zero $\gamma \in \F_{q^4}$,
$$(1,\alpha,\alpha^2,\alpha^3) + (0,\gamma,\gamma A, \gamma B)$$ has rank 3 or 4. This is equivalent to saying
that, for all $i \in \{1,\ldots,q^4-1\},$
$$(1,\alpha,\alpha^2,\alpha^3) + \alpha^i(0,1, A, B)$$
has rank 3 or 4. Using the matrix representation of vectors one can finally show that
$\hat{M}(q)$ is the number of 
$(A_1,A_2,A_3,A_4,B_1,B_2,B_3,B_4) \in \F_q^8$ such that
$$I + \begin{pmatrix}
    0 & 0 & 0 & 0 \\
    1 & 0 & 0 & 0\\
    A_1 & A_2 & A_3 & A_4 \\
    B_1 & B_2 & B_3 & B_4
\end{pmatrix} \cdot M^i$$
has rank 3 or 4 for all $i \in \{1,\ldots,q^4-1\}$, where $M$ is the companion matrix of minimal polynomial of $\alpha$ over $\F_q$. 
By combining this computational approach
with Proposition~\ref{prop:M'}
we obtain the following result.

\begin{lemma}
\label{lem:values}
The values of $M(q)$ for $q \in \{2,3,4,5,7,8\}$ are the following:
\begin{itemize}
\setlength\itemsep{0em}
    \item $M(2)=1344$,
    \item $M(3)=6368544$,
    \item $M(4)=998645760$,
    \item $M(5)=43710000000$,
    \item $M(7)=11599543859904$,
    \item $M(8)=103734668427264$.
\end{itemize}
\end{lemma}

Our next step is to use the previous values to interpolate $M(q)$ under the polynomiality assumption. The putative  degree
of $M(q)$ is too large (up to 16, as we will explain later)
for this to be possible. We therefore proceed
by constructing a polynomial of smaller degree which we can interpolate with the values computed in Lemma~\ref{lem:values}.
We start by recalling the following folklore result from elementary algebra.

\begin{lemma}
\label{lem:div}
Let $P(x), Q(x) \in \Z[x]$ be polynomials. If $\{a \in \Z \mid P(a)/Q(a) \in \Z\} =+\infty$, then $Q(x)$ divides $P(x)$ in $\Z[x]$.
\end{lemma}

By combining Lemma \ref{lem:div} with Proposition~\ref{prop:M'},
we deduce that if $M(q)$ is a polynomial in $q$, then $\hat{M}(q)$ is a polynomial in $q$ as well.
Under these assumptions, write
$$M(q) = \frac{\hat{M}(q) \, q^5 (q-1)^3 (q+1) (q^2+q+1)}{q^3-q^2-q-1}$$
and observe that
$q^5(q^3-1)(q^2-1)(q-1)$ and $q^3-q^2-q-1$ are coprime in $\Z[q]$, which implies that $q^3-q^2-q-1$ divides
$\hat M(q)$ as a polynomial. Therefore we may write $\hat M(q) = Z(q)(q^3-q^2-q-1)$, for some $Z(q) \in \Z[q]$ of degree at most 5. The bound on the degree follows from the fact that
$$M(q) \le \qbinomial{4}{2}{q^4} = q^{16} + \mbox{lower order terms in $q$},$$
which gives $\deg (M(q)) \le 16$ and thus $\deg (\hat{M}(q)) \le 8$. By definition, we have
\begin{equation*}
    Z(q) = \frac{M(q)}{q^5(q^3-1)(q^2-1)(q-1)}.
\end{equation*}
Since $Z(q)$ has degree at most 5, it can be interpolated using
the 6 values for $M(q)$ computed in Lemma~\ref{lem:values}.
The passages
are elementary and 
the final outcome is
\begin{equation*}
     Z(q) =  \frac{1}{2} \, (q^5-q^4-q^3-q^2),
\end{equation*}
which in turn gives
\begin{equation} \label{eeqq}
    M(q) = \frac{1}{2} \, q^7(q^3-1)(q^2-1)(q-1)  (q^3-q^2-q-1).
\end{equation}
Finally, by applying Proposition~\ref{prop:fandstuff} and Theorem~\ref{thm:singleton}, after lengthy computations one obtains the following result.

\begin{theorem}
\label{thm:usciamolo}
If $w_2(2,4,4;q)$ is a polynomial in $q$, then
\begin{align*}
w_0(2,4,4;q) &= 1, \\
w_1(2,4,4;q) &= -\left(q^8 + q^7 + 2q^6 - q^3 + 1\right),  \\
w_2(2,4,4;q) &= \frac{1}{2} \, \left( q^{16} + 3q^{14} + q^{13} + 3q^{12} + q^{11} + q^{10} - q^9 + q^7 + 4q^6 - 2q^3\right), \\ 
w_3(2,4,4;q) &= -\frac{1}{2}\left( q^{20} + 3q^{18} + q^{17} + 2q^{16} - q^{15} + 3q^{12} + 2q^{11} + q^{10} - q^9 - 2q^8 - q^7\right), \\
w_4(2,4,4;q) &= \frac{1}{2}\left(q^{20} + 3q^{18} + q^{17} + q^{16} - q^{15} - 3q^{14} - q^{13} + q^{11}\right).
\end{align*}
In particular, we have
\begin{multline*}
    \chi(\mL_2(4,4;q);\lambda)=(\lambda-1) (\lambda-q^4) \Biggl(\lambda^2-(q^8 + q^7 + 2q^6 - q^4 - q^3)\lambda\\+\frac{1}{2}\left(q^{16} + 3q^{14} + q^{13} + q^{12} - q^{11} - 3q^{10} - q^9 + q^7\right)\Biggr).
\end{multline*}
\end{theorem}

\begin{remark}
    It is interesting to observe that the final formula of Theorem~\ref{thm:usciamolo} is compatible with the modularity of
    $\langle e_1,e_2 \rangle_{\F_{q^4}}$ in $\mL_2(4,4;q)$, which implies that $(\lambda-1)(\lambda-q^4)$ is always a factor of
    $\chi(\mL_2(4,4;q);\lambda)$ by Theorem~\ref{thm:stanley}. We could also computationally obtain the value of $M(9)$ and it coincides with the prediction made by Equation~\eqref{eeqq}.
\end{remark}

As an application of Theorem~\ref{thm:usciamolo}, we compute the characteristic polynomial of $\mL_2(4,4;q)$ for some values of $q$, under the assumption that its second Whitney number of the first kind is a polynomial in $q$. The final results are
summarized in Table~\ref{tab:valueschi}.

\begin{center}
\renewcommand\arraystretch{1.3}
\begin{longtable}[c]{|D{0.04\linewidth}|D{0.7\linewidth}|}
\hline
    $q$   & $\chi(\mL_2(4,4;q);\lambda)$ \\\noalign{\global\arrayrulewidth 1.2pt}
    \hline
    \noalign{\global\arrayrulewidth0.4pt}
    $2$ &  $(\lambda-1)(\lambda-16)(\lambda^2 - 488\lambda + 60736)$\\\hline
    $3$ & $(\lambda-1)(\lambda-81)(\lambda^2 - 10098\lambda + 29574801)$\\\hline
    $4$ & $(\lambda-1)(\lambda-256)(\lambda^2 - 89792\lambda + 2588286976)$\\\hline
    $5$ & $(\lambda-1)(\lambda-625)(\lambda^2 - 499250\lambda + 86141640625)$\\\hline
    $7$ & $(\lambda-1)(\lambda-2401)(\lambda^2 - 6820898\lambda + 17687732901601)$\\\hline
    $8$ & $(\lambda-1)(\lambda-4096)(\lambda^2 - 19394048\lambda + 147637824126976)$\\\hline
    $9$ & $(\lambda-1)(\lambda-6561)(\lambda^2 - 48885282\lambda + 962216318765601)$ \\\hline
\caption{The characteristic polynomial of $\mL_2(4,4;q)$, under the assumption that its second Whitney number of the first kind is a polynomial in $q$. We display the unique factorization of the putative polynomial.\label{tab:valueschi}}
\end{longtable}
\renewcommand{\arraystretch}{1}
\end{center}

\section{Lattice-Rank Weights of Rank-Metric Codes} \label{sec:last}

In this section we show an application of rank-metric lattices to coding theory. 
We introduce new structural invariants (sometimes called \textit{distinguishers}) for rank-metric codes based on rank-metric lattices. We call these invariants  \textit{lattice-rank weights} and show that they extend the \textit{generalized vector rank weights} introduced in \cite{ravagnani2016generalized}. 
In our approach, the elements 
of the rank-metric lattice $\mL_i(n,m;q)$
play the same role that optimal anticodes play in~\cite{ravagnani2016generalized}.
In fact, the elements of 
$\mL_1(n,m;q)$ are precisely the optimal anticodes of~\cite{ravagnani2016generalized}.

In the second part of this section we then define the
 notions of \textit{lattice binomial moments} and \textit{lattice rank weight distribution} that are naturally associated with the invariants we introduce. We derive derive the corresponding \textit{MacWilliams-type identities} and discuss some families of rank-metric codes that are extremal with respect to these new notions.

\begin{notation}
Throughout this section we work with fixed 
$n,m\geq 2$, which we generally don't 
remember in the notation to improve readability.
We  also fix a code  $C\leq\Fqmn$ of dimension $k$ and minimum distance $d$. All dimensions are computed over $\F_{q^m}$, unless otherwise stated. 
\end{notation}

We recall the notion of equivalent codes.

\begin{definition}
    We say that the code $C,D\leq\Fqmn$ are \textbf{equivalent} if there exists an $\Fqm$-linear map $\varphi:\Fqm\rightarrow \Fqm$ such that
    $\rk(v)=\rk(\varphi(v))$ for all $\smash{v \in \F_{q^m}^n}$ and 
    $\varphi(C)=D$.
\end{definition}

Studying the invariants of codes is a natural way of investigating their structure. Several invariants are 
particularly relevant for applications due to their property of distinguishing between inequivalent codes. The latter problem is central not only in coding theory but also in some related research areas. For example, 
code distinguishers 
provide an attack tool in code-based cryptography. We propose the 
following invariants.

    \begin{definition}
       For $i\in\{1,\ldots,n\}$ and $j\in\{1,\ldots,k\}$, 
       the $(i,j)$\textbf{-th  lattice-rank weight} of $C$ is the number 
       $$\ell_j^{(i)}(C):=\min\{\dim(X):X \in \mL_i(n,m;q), \, \dim(C\cap X)\geq j\}.$$
       The $(i,j)$-\textbf{th dual lattice-rank weight} of $C$ is the number 
       $$\hat \ell_j^{(i)}(C):=\min\{\dim(X^\perp):X \in \mL_i(n,m;q), \, \dim(C\cap X)\geq j\}.$$
    \end{definition}

\begin{remark}
\label{rem:dualLi}
    If $i\in\{1,n\}$ or $m\in\{2,\ldots,i\}$, then the dual of an element of $\mL_i(n,m;q)$ is still an element of $\mL_i(n,m;q)$. However, this property does not  hold in general. Examples can be easily found.
\end{remark}

For ease of notation, we write $\smash{\ell_j^{(i)}}$, $\smash{\ell_j^{(i)\perp}}$, $\smash{\hat\ell_j^{(i)}}$ and $\smash{\ell_j^{(i)\perp}}$ instead of $\smash{\ell_j^{(i)}(C)}$, $\smash{\ell_j^{(i)}(C^\perp)}$, $\smash{\hat\ell_j^{(i)}(C)}$ and $\smash{\hat\ell_j^{(i)}(C^\perp)}$ respectively.

We also recall the following definition of generalized rank-weights proposed in~\cite{kurihara2015relative} and later characterized via anticodes in~\cite{ravagnani2016generalized}.

\begin{definition}
\label{def:mj}
    For $j\in\{1,\ldots,k\}$, the $j$-\textbf{th generalized rank weight} of $C$ is 
    \begin{equation*}
        m_j(C):=\min\{\dim(A):A\in\mathcal{A}, \,  \dim(A\cap C)\geq j\},
    \end{equation*}
    where $\mathcal{A}$ denotes the set of  $\Fqm$-subspaces of $\Fqmn$ that have a basis made of vectors with entries in $\F_q$.
\end{definition}

Again, we write $m_j$ instead of $m_j(C)$. We have the following result which, in particular, shows that the  lattice rank weights extend the generalized rank weights as a code invariant.

    \begin{theorem}
    \label{thm:propl}
        The following hold.
        \begin{enumerate}[label=(\arabic*)]
            \item\label{item1:propl} $\ell_j^{(1)}=m_j$ for all $j\in\{1,\ldots,k\}$.
            \item\label{item2:propl} $\ell_1^{(i)}=\left\lceil\frac{d}{i}\right\rceil$ for all $i\in\{1,\ldots,n\}$.
            \item\label{item3:propl} $\ell_j^{(i)}\leq n$ for all $i\in\{1,\ldots,n\}$ and $j\in\{1,\ldots,k\}$.
            \item\label{item4:propl} $\ell_j^{(i)} < \ell_{j+1}^{(i)}$ for all $i\in\{1,\ldots,n\}$ and  $j\in\{1,\ldots,k-1\}$.
            \item\label{item5:propl} $\ell_j^{(i+1)}\leq \ell_{j}^{(i)}$ for all $i\in\{1,\ldots,n-1\}$ and $j\in\{1,\ldots,k\}$.
            \item\label{item6:propl} $\ell_j^{(i)}\leq n-k+j$ for all $i\in\{1,\ldots,n\}$ and  $j\in\{1,\ldots,k\}$.
            \item\label{item7:propl} $\ell_j^{(i)}\geq \left\lceil\frac{d}{i}\right\rceil +j-1$ for all $i\in\{1,\ldots,n\}$ and $j\in\{1,\ldots,k\}$. 
        \end{enumerate}
    \end{theorem}
    \begin{proof}
        We establish the various properties separately.
        \begin{enumerate}[label=(\arabic*)]
            \item Proposition \ref{prop:L1modular} implies that the elements of $\mL_1(n,m;q)$ are precisely the $\Fqm$-subspaces of $\Fqmn$ with a basis made of vectors with entries in $\F_q$.
            The statement follows from Remark~\ref{rem:basis} and Definition~\ref{def:mj}.
            
            \item Every $v\in C$ can be expressed as the sum of no less than $\left\lceil\rk(v)/i\right\rceil$ element of rank less or equal then $i$. Hence we have
            \begin{align*}
                \ell_1^{(i)}&=\min\{\dim(X):X\in \mL_i(n,m;q)\mid\dim(C\cap X)\geq 1\}\\
                &=\min\left\{\left\lceil\frac{\rk(v)}{i}\right\rceil:v\in C\right\}=\left\lceil\frac{d}{i}\right\rceil.
            \end{align*}
        
            \item The property follows from the fact that $\dim(X)\leq n$ for any $X\in \mL_i(n,m;q)$.
        
            \item Let $X\in \mL_i(n,m;q)$ such that $\dim(X\cap C)\geq j+1$ and $\smash{\dim(X)=\ell_j^{(i)}}$. Take an element $Y\in \mL_i(n,m;q)$ such that $Y\leq X$ and $\smash{\dim(Y)=\dim(X)-1=\ell_j^{(i)}-1}$. We will prove that $\dim(Y\cap C)\geq j$. Note that, since $Y\leq X$, we have $Y\cap C=Y\cap(X\cap C)$ and $Y+(X\cap C)\leq X$. Therefore,
            \begin{align*}
                \dim(Y\cap C)&=\dim(Y\cap(X\cap C))\\
                &=\dim(Y)+\dim(X\cap C)-\dim(Y+(X\cap C))\\
                &\geq \dim(Y)+\dim(X\cap C)-\dim(X)\\
                &=\ell_j^{(i)}-1+j+1-\ell_j^{(i)},
            \end{align*}
            which implies $\dim(Y\cap C)\geq j$ and concludes the proof.
        
            \item The property follows from the definition of $(i,j)$-th lattice-rank weight.
            
            \item By \ref{item3:propl} and \ref{item4:propl} we have an increasing chain $\smash{\ell_j^{(i)}<\ell_{j+1}^{(i)}<\cdots<\ell_k^{(i)}\leq n}$. It follows that $\smash{\ell_j^{(i)}+k-j\leq n}$.
            
            \item By \ref{item2:propl} and \ref{item4:propl} we have the decreasing chain $\smash{\ell_j^{(i)}>\ell_{j-1}^{(i)}>\cdots>\ell_1^{(i)}=\left\lceil\frac{d}{i}\right\rceil}$. It follows that $\smash{\ell_j^{(i)}\geq \ell_1^{(i)}+j-1}$.
        \end{enumerate}
        
        This concludes the proof.
    \end{proof}

    \begin{remark}
        Observe that, by part~\ref{item2:propl} of Theorem~\ref{thm:propl}, we have that $\smash{\ell_1^{(i)}=1}$ for all $i\geq d$. This shows that  $\smash{\ell_1^{(i)}=\ell_1^{(i+1)}}$ for $i \ge d$ and therefore the bound in part \ref{item5:propl} of Theorem~\ref{thm:propl} is sometimes met with equality if $j=1$. Notice that the same bound can be met with equality also for $j\neq 1$. Consider, for example, $\smash{C:=\<e_1,e_2\>_{\F_{q^m}}\leq\Fqmn}$. It is not difficult to check that $\smash{\ell_1^{(i)}=1}$ and  $\smash{\ell_2^{(i)}=2}$ for all $i\in \{1,\ldots,n\}$, which shows that $\smash{\ell_2^{(i)}=\ell_2^{(i+1)}}$ for all $i\in \{1,\ldots,n\}$.
    \end{remark}
    
    In the following example we show that  
    lattice-rank weights can  distinguish between inequivalent codes when the ``fundamental'' codes parameters (length, dimension, minimum distance, and generalized weights) don't.
    
    \begin{example}
        Let $\alpha$ be the primitive element of $\F_{2^4}$ with $\alpha^4+\alpha+1=0$ and consider the  codes
        \begin{equation*}
            C:=\<(1,0,\alpha^5,\alpha),(0,1,\alpha^{14},\alpha^{13})\>, \qquad D:=\<(1,0,\alpha^6,\alpha^{10}),(0,1,\alpha^7,\alpha^3)\>,
        \end{equation*}
        where all the spans in this example are over $\F_{2^4}$. 
        The following hold.
        \begin{enumerate}[label=(\arabic*)]
            \item $\rk((1,0,\alpha^5,\alpha))=\rk((0,1,\alpha^{14},\alpha^{13}))=\rk((1,0,\alpha^6,\alpha^{10}))=\rk((0,1,\alpha^7,\alpha^3))=3$
            \item $\ell_1^{(1)}(C)=2$. This value is attained, for example, by the space $\<(1,1,0,0),(0,0,1,1)\>$ of $\mL_1(4,4;2)$, whose intersection with $C$ is $\<(1,1,\alpha^{12},\alpha^{12})\>$.
            \item $\ell_1^{(1)}(D)=2$. This value is attained, for example, by the space $\<(1,0,1,1),(0,1,0,1)\>$ of $\mL_1(4,4;2)$, whose intersection with $C$ is $\<(1,\alpha^6,1,\alpha^{13})\>$.
            \item $\ell_2^{(1)}(C)=\ell_2^{(1)}(D)=4$. This means that the only element of $\mL_1(4,4;2)$ that intersects~$C$ and~$D$ in a subspace of dimension at lest $2$ is the full space $\F_{2^4}^4$.
        \end{enumerate}
        
        Therefore, by Theorem~\ref{thm:propl}, $C$ and $D$ have the same ``fundamental'' code parameters. On the other hand, the following hold.
        
        \begin{enumerate}[label=(\arabic*)]
        \setcounter{enumi}{4}
            \item $\ell_1^{(2)}(C)=1$. This value is attained, for example, by the element $\<(1,\alpha^6,0,1)\>$ of~$\mL_2(4,4;2)$.
            \item $\ell_1^{(2)}(D)=1$. This value is attained, for example, by the element $\<(1,\alpha^6,1,\alpha^{13})\>$ of~$\mL_2(4,4;2)$.
            \item $\ell_2^{(2)}(C)=2$. This is because $C\in\mL_2(4,4;2)$, since $C=\<(1,1,\alpha^{12},\alpha^{12}),(1,\alpha^6,0,1)\>$.
            \item $\ell_2^{(2)}(D)=3$. This is because $D\notin\mL_2(4,4;2)$ and the smallest element of $\mL_2(4,4;2)$ containing $D$ is $\<(1,0,0,\alpha^{11}),(0,1,0,\alpha^{14}),(0,0,1,\alpha^8)\>$.
        \end{enumerate}
        This allows us to conclude that $C$ and $D$ are inequivalent codes, even though they share the same generalized rank weights. 
    \end{example}

    We now introduce the notions of \textit{weight distribution} and \textit{binomial moments} associated to the lattice-rank weights of a code. We show that partial information about the former provides partial information about the latter. These two notions capture structural properties of the underlying code 
    that have been investigated before 
    in the coding theory literature; see, for example, \cite{duursma2003combinatorics,blanco2018rank,byrne2020rank,byrne2021tensor}. In particular, the binomial moments have been used to define and study the \textit{zeta function} of codes; see \cite{duursma2003combinatorics} for the Hamming metric and \cite{blanco2018rank,byrne2019tensor} for the rank metric. 
    
    In the context outlined above, we also derive the MacWilliams identities for rank-metric codes associated to the lattice-rank weights. These identities provide a connection between the invariants of a code and those of its dual. 

    \begin{definition}
        Let $i\in\{1,\ldots,n\}$ and $j\in\{1,\ldots,k\}$. The $(i,j)$\textbf{-th lattice-rank distribution} of $C$ is define as the integer vector whose $(u+1)$-th component, for $u\in\{0,\ldots,n\}$, is 
        \begin{equation*}
            A_{j,u}^{(i)}(C):=\sum_{\scriptsize\begin{matrix} X\in \mL_i(n,m;q)\\ \dim(X)=u\end{matrix}} A_{j,X}^{(i)}(C),
        \end{equation*}
        where, for  $X\in \mL_i(n,m;q)$,
        $$A_{j,X}^{(i)}(C):=|\{D\leq C\cap X:\dim(D)=j \textup{ and } \nexists \;Y\in \mL_i(n,m;q) \textup{ s.t. } D\leq Y<X\}|.$$
    \end{definition}
    
    \begin{definition}
        Let $i\in\{1,\ldots,n\}$, $j\in\{1,\ldots,k\}$ and $u\in\{0,\ldots, n\}$. The $(i,j,u)$\textbf{-th  lattice binomial moment} of $C$ is defined as 
        \begin{equation*}
            B_{j,u}^{(i)}(C):=\sum_{\substack{X\in \mL_i(n,m;q)\\ \dim(X)=u}} B_{j,X}(C),
        \end{equation*}    
    where  
    \begin{equation*}
    B_{j,X}(C)=\qmbin{\dim(C\cap X)}{j} \quad \mbox{for all $X\in \mL_i(n,m;q)$.} 
    \end{equation*}
    \end{definition}
    
To further simplify the notation, in the sequel we will write
$\smash{A_{j,u}^{(i)}}$ and $\smash{B_{j,u}^{(i)}}$ instead of~$\smash{A_{j,u}^{(i)}(C)}$ and~$\smash{B_{j,u}^{(i)}(C)}$, respectively. The following result shows that $\smash{B_{j,u}^{(i)}}$ is fully determined by the code parameters, for $u$ in some intervals.

    \begin{proposition}
    \label{prop:Bju}
        We have
        \begin{equation*}
            B_{j,u}^{(i)}=\begin{cases}
            0 & \textup{ if } u<\ell_j^{(i)},\\
            \qmbin{k+u-n}{j}|\{X\in \mL_i(n,m;q):\dim(X)=u\}| & \textup{ if } u>n- \hat\ell_1^{(i)\perp}.
            \end{cases}
        \end{equation*}
    \end{proposition}
    \begin{proof}
        If $u<\ell_j^{(i)}$, then for all $X\in \mL_i(n,m;q)$ such that $\dim(X)=u$ we have
        \begin{equation*}
            \qmbin{\dim(C\cap X)}{j}=0.
        \end{equation*}
        On the other hand, if $\smash{n-u<\hat\ell_1^{(i)\perp}}$ then $C^\perp\cap X^\perp=\{0\}$ and the following holds:
        \allowdisplaybreaks
        \begin{align*}
            B_{j,u}^{(i)}(C)&=\sum_{\substack{ X\in \mL_i(n,m;q)\\ \dim(X)=u}}\qmbin{\dim(C\cap X)}{j}\\
            &=\sum_{\substack{ X\in \mL_i(n,m;q)\\ \dim(X)=u}}\qmbin{n-\dim(C^\perp+ X^\perp)}{j}\\
             &=\sum_{\substack{ X\in \mL_i(n,m;q)\\ \dim(X)=u}}\qmbin{\dim(C^\perp\cap X^\perp)+k+u-n}{j}\\
              &=\sum_{\substack{ X\in \mL_i(n,m;q)\\ \dim(X)=u}}\qmbin{k+u-n}{j}\\
            &=\qmbin{k+u-n}{j}|\{X\in \mL_i(n,m;q):\dim(X)=u\}|.
        \end{align*}
        This concludes the proof.
    \end{proof}
  
    \begin{remark}
        We have that $B_{j,u}^{(i)}=A_{j,u}^{(i)}=0$ for all $u< \ell_j^{(i)}$.
    \end{remark}
    
  The following result shows the relation between the invariants we just defined. Its proof is similar to the one of \cite[Theorem~6.4]{byrne2021tensor} and  therefore we omit it.
    
    \begin{theorem}
    \label{thm:BuAv}
        Let $i\in\{1,\ldots,n\}$, $j\in\{1,\ldots,k\}$ and $u,v\in\{0,\ldots, n\}$. Define  
        \begin{equation*}
          f_i(u,v) = |\{ (X,Y) \in \mL_i(n,m;q) \times \mL_i(n,m;q) : Y \leq X, \,  \dim(X)=u, \,  \dim(Y)=v\}|
        \end{equation*}
        and denote by $\mu^{(i)}$ the M\"obius function of the lattice $\mL_i(n,m;q)$.
        The following hold.
        \begin{enumerate}[label=(\arabic*)]
            \item 
            $\displaystyle B_{j,u}^{(i)}=\sum_{v=0}^uA_{j,v}^{(i)} \, f_i(u,v)$.
            \item $\displaystyle A_{j,v}^{(i)}=\sum_{u=0}^v\mu^{(i)}(u,v)B_{j,u}^{(i)}\, f_i(v,u)$.
        \end{enumerate} 
    \end{theorem}

The following can be seen as the MacWilliam identities for lattice-rank weights. We derive them 
 passing through the lattice binomial moments. Using Theorem~\ref{thm:BuAv}, an analogous result can be derived for the lattice-rank distribution. MacWilliam-types identities in this form have already been established in the coding theory literature in several contexts. The proof of this result is similar to the one of \cite[Theorem~6.7]{byrne2021tensor} and we therefore omit it.

\begin{theorem}
	The following holds for all $i,j\in\{1,\ldots,k\}$ and $u\in\{1,\ldots,n\}$.
	\begin{equation*}
		B_{j,u}^{(i)}=\sum_{p=0}^iq^{mp(k+u-n-i+p)}\qmbin{k+u-n}{i-p}\sum_{\substack{ X\in \mL_i(n,m;q)\\ \dim(X)=u}}\qmbin{\dim(C^\perp\cap X^\perp)}{j}
	\end{equation*}
\end{theorem}

We conclude this section by characterizing some classes of rank-metric codes that exhibit interesting rigidity properties
with respect to lattice-rank weights. The following is the vector-analogue of the family of $j$-TBMD codes; see \cite[Definition~4.1]{byrne2020rank} for matrix codes and~\cite[Definition~6.9]{byrne2021tensor} for higher-order tensors. The latter are a special class of tensor rank-metric whose \textit{tensor binomial moments} are fully determined by the code parameters.

\begin{definition}
	Let $i\in \{1,\ldots,n\}$ and $j\in \{1,\ldots,k\}$. We say that $C$ is a $(i,j)$-\textbf{LBMD} code (short for \textbf{lattice binomial moment determined}) code if 
    \begin{equation*}
        n-\ell_j^{(i)}-\hat\ell_1^{(i)\perp}<0.
    \end{equation*}
\end{definition}

As an immediate consequence of Proposition~\ref{prop:Bju} and Theorem~\ref{thm:BuAv}, we have that the lattice binomial moments and the lattice-rank distribution of $C$ associated with the lattice $\mL_i(n,m;q)$ are fully determined by its code parameters. The following two results are straighforward consequences of Theorem~\ref{thm:propl} and Proposition~\ref{prop:Bju}.

\begin{proposition}
\label{prop:LBMDimpl}
	Let $i\in \{1,\ldots,n\}$ and $j\in \{1,\ldots,k\}$. If $C$ is $(i,j)$-LBMD then $C$ is $(i-1,j)$-LBMD and $(i,j+1)$-LBMD.
\end{proposition}

\begin{proposition}
	Let $i\in \{1,\ldots,n\}$ and $j\in \{1,\ldots,k\}$. If $C$ is $(i,j)$-LBMD, then the $(i,j)$-th  lattice rank distribution and binomial moments are fully determined by $n$, $k$, $i$, and~$j$.
\end{proposition}

We now turn to the lattice-analogues of
MTR codes, which were introduced in~\cite{byrne2019tensor}.

\begin{definition}
	We say that $C\leq\Fqmn$ is \textbf{perfect} if $\smash{C=\<0\>_{\Fqm}}$ or if $C$ has a basis made of 
	rank-$1$ vectors. Equivalently, $C$ is perfect if $C\in\mL_1(n,m;q)$.
\end{definition}

 In the linear algebra literature, perfect spaces are often called 
 Frobenius-closed.

\begin{remark}
\label{rem:vtrk}
	The concept of perfect space was first introduced in \cite{atkinson1983ranks}, for matrix spaces, and later in \cite{byrne2021tensor}, for higher order tensors. As observed in \cite[Proposition~14.45]{burgisser1996algebraic} and \cite[Proposition~3.4]{byrne2019tensor}, there is a strong connection between perfect spaces and the tensor rank of a tensor code. More precisely, the tensor rank can be defined as the smallest dimension of a perfect space containing a tensor code. As a consequence, the value of  $\smash{\ell_k^{(1)}}$ can be seen as the vectorial analogue of the tensor rank of a tensor code the bound of part~\ref{item7:propl} of Theorem~\ref{thm:propl} 
	is the corresponding Kruskal-type bound; see \cite{kruskal1977three}. 
\end{remark}

In Remark~\ref{rem:vtrk}, we observed that Property~\ref{item7:propl} of Theorem~\ref{thm:propl} can be regarded as vector analogue of the tensor-rank bound. Matrix rank-metric codes meeting the Kruskal's bound with equality, namely the \textit{MTR codes}, are of great interested in coding theory for their properties of optimality in terms of storage and encoding; see~\cite[Table~1]{byrne2019tensor}. We introduce the vector analogue of this family of codes.

   \begin{definition}
        Let $i\in \{1,\ldots,n\}$. We say that the code $C$ is \textbf{$i$-lattice-optimal} if it attains the bound of part \ref{item7:propl} of Theorem \ref{thm:propl} for $j=k$, i.e., if  
        $$\ell_k^{(i)}=\ell_1^{(i)}+k-1=\left\lceil\frac{d}{i}\right\rceil+k-1.$$
    \end{definition}
 
 In the next result we show that $C$ is $i$-lattice-optimal if and only if $C$ attains the bound of part \ref{item7:propl} of Theorem \ref{thm:propl} for $i$ and for all $j\in\{1,\ldots,k\}$. In particular, the lattice-rank weights of such a $C$ indexed by $i$ are fully determined by $\smash{\ell_1^{(i)}}$.

    \begin{proposition}
    \label{prop:glwLO}
       The following statements are equivalent for all $i\in \{1,\ldots,n\}$.
        \begin{enumerate}[label=(\arabic*)]
            \item $C$ is a $i$-lattice-optimal code.
            \item $\ell_j^{(i)}=\ell_1^{(i)}+j-1$ for all $j\in \{1,\ldots,k\}$.
        \end{enumerate}
    \end{proposition}
    \begin{proof}
        If $\ell_j^{(i)}=\ell_1^{(i)}+j-1$ for all $j\in\{1,\ldots,k\}$ then, in particular, $\ell_k^{(i)}=\ell_1^{(i)}+k-1$, which implies that $C$ is $i$-lattice-optimal. On the other hand, let $C$ be $i$-lattice-optimal. By parts~\ref{item2:propl} and~\ref{item4:propl} of Theorem~\ref{thm:propl} we have an increasing chain
        \begin{equation*}            \ell_1^{(i)}<\ell_2^{(i)}<\cdots<\ell_k^{(i)}=\ell_1^{(i)}+k-1.
        \end{equation*}
        This implies the statement.
    \end{proof}
    
As a consequence of Theorem \ref{thm:propl} and \cite[Corollary~19]{ravagnani2016generalized}, we have that property~\ref{item6:propl} of Theorem~\ref{thm:propl} can be seen as a generalization of the Singleton-type bound for rank-metric codes (Theorem~\ref{thm:singleton}), which can be recovered for $i=1$ and $j=1$. This motivates the following definition.

\begin{definition}
    Let $i\in \{1,\ldots,n\}$ and $j\in \{1,\ldots,k\}$. We say that $C$ is $(i,j)$-\textbf{LMRD} (short for \textbf{lattice maximum rank distance}) if $C$ meets the bound of property~\ref{item6:propl} of Theorem~\ref{thm:propl} with equality, i.e., if
    \begin{equation*}
        \ell_j^{(i)}= n-k+j.
    \end{equation*}
\end{definition}

The following result can be seen as the vector-analogue of \cite[Theorem~4.15]{byrne2020rank}.

\begin{proposition}
\label{prop:MRD11}
    $C$ is MRD if and only if $C$ is $(1,1)$-LBMD.
\end{proposition}
\begin{proof}
    Recall that by Remark~\ref{rem:dualLi} we have that if $\{X^\perp:X\in\mL_1(n,m;q) \}=\mL_1(n,m;q)$. If $C$ is MRD then~$C^\perp$ is MRD as well~\cite{delsarte1978bilinear} and we have 
    \begin{equation*}
        n-d-d^\perp=n-(n-k+1)-(k+1)=-2<0,
    \end{equation*}
    which implies that $C$ is $(1,1)$-LBMD. On the other hand, if $C$ is $(1,1)$-LBMD then we have $n<d+d^\perp$. Therefore, we must have $d+d^\perp=n+2$ and $C$ is MRD, for example, by \cite[Corollary~18]{de2018weight}.
\end{proof}

We conclude this work with an example that illustrates some notion of optimality defined in this section.

\begin{example}
    Let $q=2$ and $n=m=4$. Consider  $C:=\<(1,0,\alpha^2,\alpha^7),(0,1,\alpha^6,\alpha^2)\>$, where the span is over $\F_{2^4}$ and $\alpha$ is a primitive element of $\F_{2^4}$ with $\alpha^4+\alpha+1=0$.
    We have  $C^\perp=\<(1,0,\alpha^6,\alpha^{10}),(0,1,\alpha^{11},\alpha^6)\>$. Recall that by Remark~\ref{rem:dualLi} we have that $\{X^\perp:X\in\mL_1(n,m;q)\}=\mL_1(n,m;q)$. The following hold.
    \begin{enumerate}[label=(\arabic*)]
        \item $\ell_1^{(1)}(C)=\ell_1^{(1)}(C^\perp)=3$, which implies that  $d(C)=d(C^\perp)=3$ by Theorem~\ref{thm:propl}. This value is attained, for example, by the element $\<(1,0,1,0),(0,1,0,0),(0,0,0,1)\>$ of $\mL_1(4,4;2)$, whose intersections with $C$ and $C^\perp$ are the span of $(1,\alpha^2,1,\alpha^3)$ and $(1,\alpha^2,1,\alpha)$, respectively.
        \item $\ell_2^{(1)}(C)=4$. This means that the only element of $\mL_1(4,4;2)$ containing $C$ is the full space $\smash{\F_{2^4}^4}$.
    \end{enumerate}
    Therefore, we have
    \begin{align*}
         4-\ell_1^{(1)}(C)-\ell_1^{(1)}(C^\perp)=-2<0, \qquad
        4-\ell_2^{(1)}(C)-\ell_1^{(1)}(C^\perp)=-3<0.
    \end{align*}
    It follows that $C$ is $(1,1)$-LBMD and $(1,2)$-LBMD. Moreover, one can easily check that $C$ is also MRD. Note that this is in line with Proposition~\ref{prop:MRD11}. 
\end{example}

\bigskip

\bibliographystyle{abbrv} 
\bibliography{latticebiblio.bib}
\end{document}